\documentclass[12pt]{amsart}
\usepackage{amscd,amsmath,amssymb,amsfonts} 
\usepackage[cmtip, all]{xy}

\newtheorem{thm}{Theorem}[section]
\newtheorem{prop}[thm]{Proposition}
\newtheorem{lem}[thm]{Lemma}
\newtheorem{cor}[thm]{Corollary}

\theoremstyle{definition}
\newtheorem{exm}[thm]{Example}
\newtheorem{defn}[thm]{Definition}

\theoremstyle{remark}
\newtheorem{remk}[thm]{Remark}
\newtheorem{remks}[thm]{Remarks}
\newtheorem{exms}[thm]{Examples}
\newtheorem{notat}[thm]{Notation}
\newtheorem{ack}{Acknowledgements}

\numberwithin{equation}{section}

{\hfill$\square$\end{defn}}

{\hfill$\square$\end{remk}}

{\hfill$\square$\end{remks}}

{\hfill$\square$\end{exm}}

{\hfill$\square$\end{exms}}

{\hfill$\square$\end{notat}}

\newcommand{\sC}{{\mathcal C}}

\newcommand{\sF}{{\mathcal F}}
\newcommand{\sG}{{\mathcal G}}

\newcommand{\sN}{{\mathcal N}}
\newcommand{\sO}{{\mathcal O}}

\newcommand{\sZ}{{\mathcal Z}}
\newcommand{\wt}{\widetilde}

\newcommand{\A}{{\mathbb A}}

\newcommand{\C}{{\mathbb C}}

\newcommand{\G}{{\mathbb G}}

\newcommand{\Q}{{\mathbb Q}}
\newcommand{\R}{{\mathbb R}}

\newcommand{\Z}{{\mathbb Z}}

\newcommand{\fm}{{\mathfrak m}}
\newcommand{\fg}{{\mathfrak g}}

\newcommand{\inj}{\hookrightarrow}

\newcommand{\Wedge}{{\Lambda}}

\newcommand{\ds}{{/\kern-3pt/}}

\newcommand{\ov}{\overline}
\newcommand{\wh}{\widehat}

\begin{document}
\title{Completions of Higher Equivariant $K$-theory}
\author{Amalendu Krishna}


\baselineskip=10pt 
                                              
\begin{abstract}
For a linear algebraic group $G$ acting on a smooth variety $X$ over an
algebraically closed field $k$ of characteristic zero, we prove 
a version of non-abelian localization theorem for the rational
higher equivariant $K$-theory of $X$. This is then used to 
establish a Riemann-Roch theorem for the completion of the
rational higher equivariant K-theory at a maximal ideal of
the representation ring of $G$. 
\end{abstract}

\maketitle   

\section{Introduction}
Ever since Grothendieck proved his very general Riemann-Roch theorem,
the Riemann-Roch problem has by now become one of the most important tools
in algebraic geometry for studying the sheaves and bundles on varieties
and analytic spaces in terms of the various cohomology theories on the
variety such as the motivic and the singular cohomology. The well known
Riemann-Roch theorem of Baum, Fulton and MacPherson \cite{BFM} states 
that the Grothendieck group of coherent sheaves on an algebraic variety
has a natural isomorphism with the Chow groups of algebraic cycles on
the variety, when these groups are taken with the rational coefficients.
This result was later generalized by Bloch \cite{Bloch} when he founded
the theory of higher Chow groups and showed that these higher Chow groups
form a cohomology theory. Moreover, when considered with the 
rational coefficients, they completely describe the higher K-theory of 
coherent sheaves of a quasi-projective variety. 

When a linear algebraic group $G$ acts on a variety $X$, it becomes very 
natural to ask if the Grothendieck group or more generally, the $K$-theory 
of the equivariant coherent sheaves has similar description in terms of the 
geometry of the variety. Such a description becomes important even from the 
point of view of the representation theory of the group and gives a more 
concrete and geometric way of understanding the representation theory of 
linear algebraic groups. However, it turns out, as is evident from the work 
of Edidin and Graham \cite{ED0}, that if at all there is one, such a 
description can not be as direct as in the non-equivariant case. The basic 
reason for this is that the equivariant (higher) Chow groups are always 
complete with respect to certain linear topology while this is certainly not 
the case with the equivariant K-theory.

A {\sl variety} in this paper will mean a reduced, connected and
separated scheme of finite type with an ample line bundle over a field $k$.
This base field $k$ will be fixed throughout this paper
and will be assumed to be algebraically closed and of characteristic zero
(except in Section~2, where $k$ is arbitrary).
Let $G$ be a connected linear algebraic group over $k$ acting on such a 
variety $X$. Recall that this action on $X$ is said to be linear if $X$ 
admits a $G$-equivariant ample line bundle, a condition which is always 
satisfied if $X$ is normal \cite[Theorem~2.5]{Sumihiro}. All $G$-actions in 
this paper will be assumed to be linear.  
For $i \ge 0$, let $K^G_i(X)$ (resp. $G^G_i(X)$) denote the $i$th 
homotopy group of the $K$-theory spectrum of $G$-equivariant vector 
bundles (resp. coherent sheaves) on $X$. The $G$-equivariant higher Chow
groups $CH^j_G(X, i)$ of $X$ were defined by Edidin and Graham 
({\sl cf.} \cite{ED2}, also see Section~5 for more detail) as the ordinary 
higher Chow groups of the quotient space $X \stackrel{G}{\times} U$, where
$U$ is an open subscheme of a representation of $G$ on which the action of
$G$ is free, and its complement is of sufficiently high codimension.
The representation ring $R(G)$ of $G$ is simply the group $K^G_0(k)$. 
The ring structure is given by the tensor
product of representations. In this paper, all the (equivariant) 
$K$-groups and higher Chow groups will be considered with the rational
coefficients. In particular, the representation ring $R(G)$ will actually
mean $R(G){\otimes}_{\Z} {\Q}$. For any $\Q$-algebra $A$ and a $\Q$-vector 
space $V$, the group $V_A$ will mean the group $V {\otimes}_{\Q} A$. These 
notations will be very often employed for the equivariant $K$-groups and higher
Chow groups. 

In view of the completeness of the product of equivariant Chow groups as 
stated above, Edidin and Graham showed in {\sl loc. cit.} that if 
$G^G_0(X)$ is completed with respect to the maximal ideal of the the 
virtual rank zero representations in $R(G)$, then this completion is 
isomorphic to the product of the equivariant Chow groups via a 
Riemann-Roch map. This was the first result of Riemann-Roch type in the
equivariant $K$-theory. Subsequently in \cite{ED3}, Edidin and Graham
studied the problem of describing the completion of $G^G_0(X)$ with
respect to any given maximal ideal of $R(G)$ in terms of algebraic cycles,
in the case when the base field is $\C$ and all the Grothendieck groups,
Chow groups and the representation rings are considered with the complex 
coefficients. 
To state their result in a concise form, recall from \cite{ED1} that if
$G$ is a complex algebraic group, every maximal ideal 
${\mathfrak m}_{\Psi}$ of $R_{\C}(G)$ is given as the ideal of the 
representations whose characters vanish at a unique semi-simple conjugacy 
class $\Psi$ in $G$. 
\begin{thm}[Edidin-Graham]\label{thm:ED} Let $G$ act on a complex 
algebraic space $X$. For any maximal ideal ${\mathfrak m}_{\Psi}$ of
$R_{\C}(G)$, let $Z$ denote the centralizer of an element $h$ of 
$\Psi$ and let $X^h$ be the fixed subspace of $X$ for $h$.
Then there is a Riemann-Roch isomorphism
\[
{\tau}^{\Psi}_X : 
{\widehat{\left(G^G_0(X)_{\C}\right)}}_{{\mathfrak m}_{\Psi}} 
\to \stackrel{\infty}{\underset {j=0}{\prod}}{CH^j_Z(X^h)}_{\C}.
\]    
Moreover, for smooth $X$, this map is described in terms of 
Chern characters and Todd classes.
\end{thm}
Edidin and Graham prove this result by extending Thomason's 
localization theorem \cite{Thom0} for the algebraic
$K$-theory for the action of diagonalizable groups to the case of
general complex algebraic groups ({\sl cf. loc. cit.}, Theorem~4.3).
This localization theorem implies in particular that in the situation
of Theorem~\ref{thm:ED}, the restriction map of the completions of
the equivariant $K$-theory induces an isomorphism
\begin{equation}\label{eqn:EDLoc}
i^{!} : {\widehat{\left(G^G_i(X)_{\C}\right)}}_{{\mathfrak m}_{\Psi}}    
\to {\widehat{\left(G^Z_i(X^h)_{\C}\right)}}_{h},
\end{equation} 
where the term on the right is the completion of the equivariant $K$-groups
with respect to the maximal ideal corresponding to the conjugacy class
in $Z$ consisting of the single element $h$.

The results of this paper were motivated by some of the questions raised in
\cite{ED3}. The first natural question is the formulation of the 
Edidin-Graham's completion theorem for the rational $K$-theory rather
than the complex $K$-theory of varieties with group actions. The
second and the more important question is whether it is possible to describe
the completion of the higher equivariant $K$-theory with respect to
the various maximal ideals of $R(G)$ in terms of the
algebraic cycles on some subspaces. Our main purpose in this paper is to answer
these questions for the rational equivariant higher $K$-theory of
smooth varieties.

We fix some notations before we state our first result. For any
finitely generated abelian group $N$ and any field $l$ of characteristic
zero, let $T = D_l(N)$ denote the unique split diagonalizable group over $l$ 
whose character group is given by $N$. $D_l(N)$ is the spectrum of the group 
algebra $l[N]$. It is well known that $D_l(N)$ is a torus if and only if $N$ 
has no torsion. There is also a canonical isomorphism of $\Q$-algebras
${\Q}[N] \to R(T)$. By \cite[Lemme~1.1, Proposition~1.2]{Thom0}, every
maximal ideal $\mathfrak m$ of $R(T)$ corresponds to a unique 
diagonalizable closed subgroup $T_{\fm} = D_{l}(N/{N_{\fm}})$ of $T$ where
$N_{\fm}$ is the kernel of the natural map of abelian groups
\[
N \xrightarrow{\phi} {\left({R(T)}/{\fm}\right)}^* \ {\rm given \ by} \ 
n \mapsto [n] \ mod \ {\fm}.
\]
$T_{\fm}$ is called the {\sl support} of $\fm$. Note that $T_{\fm}$ is 
connected if and only if the image of $\phi$ has no torsion. 
Let $G$ be a connected linear algebraic group over $k$ with a fixed maximal 
torus $T = D_k(N)$ of rank $r$. It follows from Proposition~\ref{prop:finiteR}
below that the natural map $R(G) \to R(T)$ of $\Q$-algebras is finite.
In particular, every maximal ideal $\fm$ of $R(G)$ can be lifted to a 
(not necessarily unique) maximal ideal $\wt{\fm}$ of $R(T)$. There is a
distinguished maximal ideal $I_G$ of $R(G)$, which is the ideal of the rank 
zero virtual representations, i.e., $I_G$ is the kernel of the rank map 
$R(G) \to \Q$. It is called the augmentation ideal of $R(G)$.
It is easy to see that ({\sl cf.} Corollary~\ref{cor:maximal2}) $I_T$ is the
only lift of $I_G$ in $R(T)$ and the support of $I_T$ is the identity
subgroup of $T$.

Suppose now that $G$ acts linearly on a smooth quasi-projective variety
$X$ over $k$. It is then well known that the natural map $K^G_*(X) \to 
G^G_*(X)$ is an isomorphism.
Hence we shall not make any distinction between these groups for smooth
varieties in this paper. It is also well known that $K^G_*(X)$ and
$G^G_*(X)$ are naturally modules for the $\Q$-algebra $K^G_0(X)$ and hence 
they are also modules for the $\Q$-algebra $R(G)$. For any maximal ideal 
$\fm$ of a commutative ring $A$ and for any $A$-module $M$, let 
${\widehat{M}}_{\fm}$ denote the $\fm$-adic completion of $M$.
We fix a maximal ideal $\fm$ of $R(G)$ and choose its lift $\wt{\fm}$ as
a maximal ideal of $R(T)$. Let $T_{\wt{\fm}} \subset T$ denote the support 
of $\fm$ and let $Z = Z_G(S)$ denote the centralizer of 
$S = {T_{\wt{\fm}}} {\otimes}_{\Q} k$ in $G$ under the inclusion 
$S \subset T \subset G$.
Let $\ov{\fm}$ be the restriction of $\wt{\fm}$ to $R(Z)$ under the natural
restriction maps $R(G) \to R(Z) \to R(T)$.
Our first main result deals with the problem of representing the  
completions ${\wh{K^G_i(X)}}_{\fm}$ of the equivariant higher $K$-theory
of $X$ in terms of algebraic cycles. For such a variety $X$, put
\[CH^*_G(X,i) = {\underset {j \ge 0}{\bigoplus}} {CH^j_G(X,i)} {\otimes}_{\Z}
{\Q} \ \ {\rm for} \ \
i \ge 0.\]
It is known \cite{ED2} that the term on the right is an infinite sum in 
general. Put $S(G) = CH^*_G(k, 0)$. Then $S(G)$ is a graded $\Q$-algebra 
with $S(G)_0 = CH^0_G(k) = {\Q}$, where the graded ring structure is given 
by the intersection product on the equivariant Chow groups of the classifying 
space $BG$. Moreover, this intersection product and the pullback maps on the 
equivariant Chow groups make $CH^*_G(X, i)$ a graded module for the ring 
$S(G)$. Let $J_G$ denote the maximal ideal 
${\underset {j \ge 1}{\bigoplus}} {CH^j_G(k,0)}$ and let
$\wh{CH^*_G(X,i)}$ denote the $J_G$-adic completion of the $S(G)$-module
$CH^*_G(X,i)$.   
It was shown in \cite[Theorem~1.2]{KV} that if $G$ acts on a smooth
quasi-projective variety $X$, then there is a Chern character
map 
\begin{equation}\label{eqn:KV}
{ch}^X : {\wh{K^G_i(X)}}_{I_G} \to \wh{CH^*_G(X,i)}
\end{equation}
which is an isomorphism. This result was proved for $i = 0$ by Edidin and
Graham \cite[Theorem~4.1]{ED0}. This shows that the completion of the
higher $K$-groups at the augmentation ideal of $R(G)$ can be represented by 
the equivariant higher Chow groups via Chern characters. Our purpose here is 
to give a similar representation of the completion of the rational
equivariant higher $K$-groups at other maximal ideals of the representation 
ring. This in particular extends the result of \cite[Theorem~5.5]{ED3} to the
higher equivariant $K$-theory of smooth varieties and hence answers
a question of Edidin and Graham in {\sl loc. cit.}. The following result
in particular shows that although it is impossible to prove the direct 
equivariant analogue of the Riemann-Roch Theorem of Bloch 
\cite[Theorem~9.1]{Bloch} showing the isomorphism between the rational higher 
$K$-theory and higher Chow groups, it is indeed possible to prove such an 
isomorphism whenever the equivariant higher $K$-groups are completed at a 
given maximal ideal of the representation ring of the underlying group. 
\begin{thm}\label{thm:TRR}
Let the group $G$ act on a smooth quasi-projective variety $X$
and let $\fm$ be a given maximal ideal of $R(G)$. Assume that the support
$T_{\wt{\fm}}$ is a subtorus of the maximal torus $T$ of $G$. 
Let $L = {R(T)}/{\wt{\fm}}$ be the residue field of $\wt{\fm}$ and
let $f: X^{\wt{\fm}} \inj X$ denote the inclusion of the fixed locus for
the action of $T_{\wt{\fm}}$ on $X$. Then for every $i \ge 0$, there is a 
Chern character map
\[
{ch}^X_{\fm} : {\wh{K^G_i(X)}}_{\fm , L} \to
\wh{CH^*_Z(X^{\wt{\fm}},i)}_L
\]
which is an isomorphism. 
\end{thm}
As in the proof of the twisted Riemann-Roch theorem in {\sl loc. cit.}
for the Grothendieck group of equivariant coherent sheaves, we deduce 
the above result by first proving the following version of the non-abelian
completion theorem for the rational equivariant $K$-theory. We first 
generalize the completion result of Edidin and Graham for the complex 
equivariant $K$-theory to the case of $K$-theory with coefficients in the 
base field $k$. This is then used to prove the following result for the
rational equivariant $K$-theory by an extension and descent argument. 
To state the result, we fix a maximal ideal 
$\fm$ of $R(G)$ with a lift $\wt{\fm}$ as a maximal ideal of $R(T)$. 
Let $\ov{\fm}$ be the restriction of $\wt{\fm}$ to $R(Z)$ in the notations
of Theorem~\ref{thm:TRR}. We have the following restriction map of various 
completions.
\begin{equation}\label{eqn:res}
{res}_{\fm} : {\wh{K^G_i(X)}}_{\fm} \to {\wh{K^Z_i(X)}}_{\fm} \to
{\wh{K^Z_i(X)}}_{\ov{\fm}}.
\end{equation}  
\begin{thm}\label{thm:NAL}
Let $G$ act on a smooth quasi-projective variety $X$ and let
$\fm$ be a given maximal ideal of $R(G)$. Assume that the support
$T_{\wt{\fm}}$ is a subtorus of the maximal torus $T$ of $G$ and 
let $f: X^{\wt{\fm}} \inj X$ denote the inclusion of the fixed locus for
the action of $T_{\wt{\fm}}$ on $X$.
Then for every $i \ge 0$, all the horizontal arrows in the commutative
diagram
\begin{equation}\label{eqn:NAL*}
\xymatrix{
{\wt{f}}^!_{\fm} : K^G_i(X) {\otimes}_{R(G)} {\wh{R(G)}}_{\fm}
\ar[d] \ar[r]^{{\wt{res}}_{\fm}} & 
K^Z_i(X) {\otimes}_{R(Z)} {\wh{R(Z)}}_{\ov{\fm}}
\ar[r]^{{\wt{f}}^*} & K^Z_i(X^{\wt{\fm}}) 
{\otimes}_{R(Z)} {\wh{R(Z)}}_{\ov{\fm}} \ar[d] \\
f^!_{\fm} :  {\wh{K^G_i(X)}}_{\fm} \ar[r]_{{res}_{\fm}} &
{\wh{K^Z_i(X)}}_{\ov{\fm}} \ar[r]_{f^*} & 
{\wh{K^Z_i(X^{\wt{\fm}})}}_{\ov{\fm}}}
\end{equation}  
are isomorphisms of $\wh{R(G)}_{\fm}$-modules and the vertical arrows
are injective.
\end{thm}
We now make some remarks about the above results. We first point out
that our results are stated for the action of reductive groups.
However, the general case can be easily deduced from this by choosing
a Levi subgroup $L$ containing $T$, which always exists in 
characteristic zero as can be seen using the Lie algebras. The proofs are then
completed by Proposition~\ref{prop:NR}. We do not say more on this.
The more important remark we want to make is that our version of non-abelian 
completion and Riemann-Roch theorems are proved for the completion of the 
rational equivariant $K$-theory at the maximal ideals
of the representation ring. As the reader would observe, many of the 
intermediate results proved in this paper hold for all prime ideals of the 
representation ring although they have been stated only for the maximal
ideals. We expect that the main results of this paper can be proved for 
all prime ideals using these intermediate results and some more 
analysis.  
The reader can also see that even though the non-abelian completion
holds for the rational $K$-theory, one has to replace the field of rationals
by the residue field of the given maximal ideal for proving the Riemann-Roch 
theorem. In fact, our proof of Theorem~\ref{thm:TRR} (especially 
Lemma~\ref{lem:twist}) would show that it may not be possible to avoid doing 
this if one wants to represent the equivariant higher $K$-theory by the 
algebraic cycles.

We end the introduction with a brief outline of the various sections of this
paper. As in \cite{ED3}, our proof of the Riemann-Roch theorem 
for equivariant higher $K$-theory is based on the non-abelian completion
theorem for the rational equivariant $K$-theory. In order to prove this,
we devote our next section to the study of the geometric aspects of the 
variety defined by the representation ring of the underlying algebraic group 
$G$ defined over an arbitrary field $k$ (of any characteristic).
This helps in particular to determine all the maximal ideals of $R(G)$ as well
as $R_k(G)$. The results of this section do not depend on the nature of the
field $k$ and hence can be helpful in pursuing some of the questions 
discussed above. Section~3 is devoted to proving most of the preliminary
results one needs to prove Theorem~\ref{thm:NAL}. We finally prove this theorem
in Section~4 by first proving it for those groups whose derived groups
are simply connected and then deducing the general case from this through
some reduction steps. Section~5 is devoted to the proof of our Riemann-Roch
theorem. We do this by first defining an automorphism of the higher
equivariant $K$-theory which converts the completion with respect to a given
maximal ideal to the completion with respect to the augmentation ideal of 
the representation ring. This is then combined with Theorem~\ref{thm:NAL}
and a version of Riemann-Roch theorem for the higher equivariant $K$-theory, 
proved recently in \cite{KV}. 
\section{Geometry of the Representation Ring $R_k(G)$}
In this section, $k$ will denote any field with arbitrary characteristic.
Let $G$ be a linear algebraic group over a field $k$.
Recall that $R_k(G)$ denotes the representation
ring of $G$ with coefficients in $k$, i.e., $R_k(G) = R(G) {\otimes}_{\Z} k$.
We establish the preliminary results about $R_k(G)$ in this section. Our main
purpose here is to identify this $k$-algebra as the invariant subalgebra
of the coordinate ring of the group $G$ for the adjoint action of $G$ on 
itself. This is then used to describe the maximal ideals of $R_k(G)$ in terms 
of the closed points of the group $G$ itself.
Let $k$ be any given field. We begin with the following finiteness 
result about the maps between the integral representation rings. When $k$ is 
the field of complex numbers, such a result was proved for the rings
$R_{\C}(G)$ in \cite{ED1} using the corresponding result of Segal 
\cite{Segal} for the compact Lie groups.  
\begin{prop}\label{prop:finiteR}
Let $G$ be a connected linear algebraic group over $k$ and let $H$ be a
closed connected subgroup of $H$. Then the restriction map of integral 
representation rings $R(G) \to R(H)$ is finite.
\end{prop}
\begin{proof} We first prove the proposition when $k$ is algebraically
closed and $H$ is a maximal torus $T$ of $G$. If $V$ is an irreducible 
representation of $G$ and if the unipotent radical $R_uG$ of $G$ acts 
non-trivially on $V$, then the invariant subspace of $V$ for this action 
is a non-zero subspace ({\sl cf.} \cite[Corollary~10.5]{Borel})
which is invariant under $G$ as well since 
$R_uG$ is normal in $G$. This contradicts the ireducibility of $V$. Hence
$R_uG$ must act trivially on $V$. Since $R(G)$ is a free abelian group
on the set of irreducible representations of $G$, we see that the
map $R(G) \to R(G/{R_uG)}$ is an isomorphism. Hence we can assume that 
$G$ is reductive. Let $W$ denote the Weyl group of $G$. Then $W$ acts
naturally on $T$ by inner automorphisms which induces an action of $W$ on
$R(T)$. Moreover, \cite[Th{\'e}oreme~4]{Serre} implies that
and the natural map of the representation rings induces a ring isomorphism
\begin{equation}\label{eqn:inv}
R(G) \xrightarrow{\cong} {\left(R(T)\right)}^W,
\end{equation}
where the term on the right is the ring of invariants for the Weyl
group action. Since the natural map ${\Z}[N] \to R(T)$ is an isomorphism
of rings (where $N$ is the character group of $T$), we see that $R(T)$
is a finite type algebra over $\Z$. In particular, it is finite over 
${\left(R(T)\right)}^W$ which is then also a finite type $\Z$-algebra 
({\sl cf.},\cite[Lemma~4.7]{KV}) and hence noetherian.
Now if $H$ is a closed connected subgroup of $G$, we choose a maximal
torus $S$ of $H$. Since $k$ is algebraically closed, there is a maximal
torus $T$ of $G$ containing $S$. Then we get a commutative diagram
of representation rings
\[
\xymatrix{
R(G) \ar[r] \ar[d] & R(T) \ar[d] \\
R(H) \ar[r] & R(S)}
\]
We have shown above that the horizontal arrows are finite and injective 
maps. Since $S$ is a subtorus of $T$, there is a decomposition $T = S 
\times S'$ and hence the map $R(T) \to R(S)$ is surjective. In particular,
$R(S)$ is finite over $R(G)$. Since we have also shown above that $R(G)$
is noetherian, we conclude that $R(H)$ is finite over $R(G)$.

Finally, we consider the case when $k$ is not necessarily algebraically
closed. We can embed $G$ as a closed subgroup of $GL_n$ over $k$ for 
some $n$ and then it suffices to show that $R(H)$ is finite over $R(GL_n)$.
For any algebraic group $G$ over $k$, let $G_{\ov k}$ denote the base change
to the algebraic closure of $k$. Then we get the following commutative
diagram of representation rings.
\[
\xymatrix{
R(GL_n) \ar[r] \ar[d] & R(GL_{n, \ov{k}}) \ar[d] \\
R(H) \ar[r] & R(H_{\ov k})}
\]
By the Galois descent for the rational algebraic $K$-theory, it is easy
to see that when tensored with the field of rational numbers,
the terms on the left are the Galois invariants of the corresponding
terms on the right ({\sl cf.} \cite[Lemma~8.4]{Gubel}). In particular,
the kernels of the horizontal maps are torsion. On the other hand,
a representation ring is the free abelian group on the set of irreducible
representations and hence has no torsion. We conclude that $R(H)$ is a 
$R(G)$-submodule of $R(H_{\ov k})$. Since $R(GL_n)$ is known be noetherian,
we only now need to show that $R(H_{\ov k})$ is finite over $R(GL_n)$.
But this is immediate from the case of algebraically closed fields and
the fact that $R(GL_n) \cong  R(GL_{n, \ov{k}})$.
\end{proof}  
\begin{remk}\label{remk:SHfinite}
It is shown in \cite{KV} that for the groups $G$ and $H$ as
in Proposition~\ref{prop:finiteR}, map $S(G) \to S(H)$ of the equivariant 
Chow rings is finite with the rational coefficients. It is still not known
if this holds with the integral coeffients as well.
\end{remk} 

We next state a completely known result about the semi-simple conjugacy 
classes in a connected reductive group.
\begin{lem}\label{lem:SSimple}
Let $T$ be a maximal torus of a connected reductive group $G$ over an
algebraically closed field $k$. Let $N$ denote the normalizer of $T$ in
$G$. Then any two elements of $T$ which are conjugate in $G$, are also 
conjugate by an element of $N$. In particular, the set of semi-simple conjugacy
classes in $G$ is in natural bijection with the set $T/W$ of $W$-orbits in $T$,
where $W$ is the Weyl group of $G$.
\end{lem} 
\begin{proof}(Sketch) Suppose $s, t \in T$ are such that $t = hsh^{-1}$ in
$G$. It is then easy to check directly that
$hTh^{-1}$ is contained in $Z_G(t)$ and hence $T$ and $hTh^{-1}$ are 
maximal tori of the identity component of $Z_G(t)$. Hence they are conjugate
by an element $g$ of $Z_G(t)$ ({\sl cf.} \cite[Corollary~11.3]{Borel}). But 
this implies that $gh \in N$ and $(gh)s{(gh)}^{-1} = gtg^{-1} = t$.  
 
For the second assertion, first note that every semi-simple element lies in a
maximal torus and hence its conjugacy class is same as the conjugacy class
of an element of $T$. Moreover, the above assertion shows that two such
conjugacy classes are distinct if and only if they are not conjugate by
an element of $N$. On the other hand, $W = N/T$ acts on $T$ by conjugation.
This implies that $G$-conjugacy classes of two elements of $T$ are distinct
if and only if they are in distinct $W$-orbits.
\end{proof} 

Let $G$ be a connected reductive group over a field $k$.
For any affine $k$-variety $X$, let $k[X]$ denote the coordinate ring of $X$. 
Let ${Sch}_k$ denote the category of finite type $k$-schemes. Note then that
there is a natural isomorphism of $k$-algebras
\begin{equation}\label{eqn:CRing}
k[X] \xrightarrow{\cong} {Hom}_{{Sch}_k}(X, {\A}^1_k).
\end{equation}
Recall that the algebraic group $G$ is given by the data of $k$-morphisms
\[
G {\times}_k G \xrightarrow{\mu} G, \ \ G \xrightarrow{\iota} G, \ \
Spec(k) \xrightarrow{e} G,
\]
which define the multiplication, inverse and the identity section
respectively. This data is equivalent to the following data of
morphisms of the Hopf algebras over $k$.
\[
k[G] \xrightarrow{\wh{\mu}} k[G] {\otimes}_k k[G], \ \ 
k[G] \xrightarrow{\wh{\iota}} k[G], \ \
k[G] \xrightarrow{\wh{e}} k.
\]
Let $G \xrightarrow{\Delta} G \times G$ denote the diagonal map and
let $G \times G \xrightarrow{sw} G \times G$ denote the map which
switches the coordinates.  
The {\sl adjoint} action of $G$ on itself by the composite map
`ad' given by
\begin{equation}\label{eqn:adj}
G {\times} G \xrightarrow{(\Delta, id)} G {\times} G {\times} G
\xrightarrow{(id, sw)} G {\times} G {\times} G \xrightarrow{(\mu, id)}
G {\times} G \xrightarrow{\mu} G.
\end{equation} 
In terms of the dual action of Hopf algebras, we denote it by
\begin{equation}\label{eqn:adj1} 
k[G] \xrightarrow{\wh{ad}} k[G] {\otimes}_k k[G].
\end{equation}
It is clear from the definition that when $k$ is algebraically closed,
the action is indeed the classical adjoint action $(g, h) \mapsto ghg^{-1}$
of $G$ on itself. To show that ~\ref{eqn:adj} indeed defines an
action, one needs to check that the following diagrams of the maps
of Hopf algebras commute.
\[
\xymatrix{
k[G] \ar[r]^{\wh{ad}} \ar[d]_{\wh{ad}} & k[G] {\otimes}_k k[G] 
\ar[d]^{{\wh{ad}} \otimes id} \\
k[G] {\otimes}_k k[G] \ar[r]^{id \otimes {\wh{ad}}} &
k[G] {\otimes}_k k[G] {\otimes}_k k[G]}
\]
\[
\xymatrix{
k[G] \ar@/_1pc/[rr]_{id} \ar[r]^{\wh{ad}} &
k[G] {\otimes}_k k[G] \ar[r]^{{\wh{e}} \otimes id} & k[G]}  
\]
However, the injectivity of the natural map $Hom_{sch_k} (X, Y) \to
Hom_{sch_{\ov k}} (X_{\ov k}, Y_{\ov k})$ (which can be checked locally)
implies that it is enough to check the commutativity when $k$ is algebraically
closed. But this is a straightforward checking since the adjoint
action in this case is just the conjugation.

We recall from \cite[Section~1]{GIT} that an action of $G$ on an affine
$k$-scheme $X$ is equivalent to the map of $k$-algebras
\begin{equation}\label{eqn:action}
\wh{\alpha} : k[X] \to  k[G] {\otimes}_k k[X]
\end{equation}
which satisfies the commutativity of certain diagram of maps of
$k$-algebras similar to the above. We refer the reader to {\sl loc. cit.},
Section~1 for more detail. 
\begin{defn}\label{defn:inv}
We shall say that an element $f$ of $k[X]$ is $G$-invariant for this action 
if ${\wh{\alpha}}(f) = 1 {\otimes} f$. 
\end{defn}
It is easy to check that the set of $G$-invariant elements in $k[X]$ is
a $k$-sub-algebra and will be denoted by ${k[X]}^G$. This is the
coordinate ring of the universal geometric quotient of $X$ for the action of 
$G$ on $X$. In particular, if $k$ is algebraically closed, then ${k[X]}^G$
is the invariant elements in $k[X] = Hom_{sch_k}(X, {\A}^1_k)$ for the 
action of $G(k)$ on $k[X]$ given by 
\begin{equation}\label{eqn:action1}
G \times k[X] \to k[X] 
\end{equation}
\[
(g, f) \mapsto f^g,
\]
where $f^g(x) = f \circ {\alpha}\left((g,x)\right)$.

We now consider the adjoint action of $G$ on itself as described in
~\ref{eqn:adj}. In this case, we shall often write ${k[X]}^G$
as $C[G]$. In order to understand the geometry of the representation ring
of $G$, our first step is to define a $k$-algebra map from $R_k(G) \to C[G]$ 
which has some interesting properties.  
So let $(V, \rho)$ be an $n$-dimensional representation of $G$ given
by the morphism of $k$-algebraic groups $G \xrightarrow{\rho} GL(V)$.
Let ${\chi}: GL(V) \subset {\rm End}(V) \to {\A}^1_k$ denote the
character morphism described algebraically by the composite map
\[
k[t] \xrightarrow{\wh{\chi}} k[X_{ij}] \inj k[X_{ij}, 1/{\rm det}],
\]
\[
{\wh{\chi}}(t) = \stackrel{n}{\underset {i=0}{\Sigma}}{X_{ii}}.
\]
Composing this map with $\rho$, we get the character map
\begin{equation}\label{eqn:character} 
k[t] \xrightarrow{{\wh{\chi}}_{\rho}} k[G].
\end{equation}
It is easy to check from above that if $(V_1, {\rho}_1)$ and 
$(V_2, {\rho}_2)$ are two representations of $G$, then
${{\wh{\chi}}_{{\rho}_1 \oplus {\rho}_2}}(t) =
{{\wh{\chi}}_{{\rho}_1}}(t) + {{\wh{\chi}}_{{\rho}_2}}(t)$
and ${{\wh{\chi}}_{{\rho}_1 \otimes {\rho}_2}}(t)
= {{\wh{\chi}}_{{\rho}_1}}(t) \cdot {{\wh{\chi}}_{{\rho}_2}}(t)$.
In other words, the assignment $(V, \rho) \mapsto {\wh{\chi}}_{{\rho}}(t)$
induces a $k$-algebra morphism $R_k(G) \xrightarrow{{\phi}_G} k[G]$.
\begin{prop}\label{prop:character1}
Let $G$ be a connected reductive group over a field $k$
such that it admits a maximal torus $T$ which is split over $k$. Then
the map ${\phi}_G$ induces an isomorphism of $k$-algebras
\begin{equation}\label{eqn:char}
{\phi}_G : R_k(G) \to C[G].
\end{equation}
\end{prop}
\begin{proof} We first have to show that the image of
${\phi}_G$ is contained in $C[G]$. 
For this, it suffices to show that for any representation $(V, \rho)$,
one has $\wh{ad} \circ {\wh{\chi}}_{\rho} = 1 \otimes 
{\wh{\chi}}_{\rho}$. Since the map $k[G] {\otimes}_k k[G] \to
{\ov k}[G_{\ov k}] {\otimes}_{\ov k} {\ov k}[G_{\ov k}]$ is injective,
it suffices to show that $\wh{ad}_{\ov k}\circ{\wh{\chi}}_{{\rho}_{\ov k}}
= 1 \otimes {\wh{\chi}}_{{\rho}_{\ov k}}$. Thus we can assume that
$k$ is algebraically closed. In this case, we can use ~\ref{eqn:action1} 
and the discussion following ~\ref{eqn:adj1}
to reduce to showing that ${\chi}_{\rho}(ghg^{-1}) =
{\chi}_{\rho}(h)$. But this follows immediately from the standard
properties of the trace of matrices.

To prove the isomorphism of ${\phi}_G$, note that as $G$ is $k$-split,
it is uniquely described by a root system over $k$ and hence its Weyl
group $W$ is a constant finite group which does not depend on the
base change of $G$ by any field extension of $k$. Moreover, $W$ acts on
$T$ and hence on $R_k(T)$ and $C[T] = k[T]$ such that the map ${\phi}_T$
is $W$-equivariant. Thus we have the following commutative diagram.
\begin{equation}\label{eqn:char1}
\xymatrix{
R_k(G) \ar[r]^{{\phi}_G} \ar[d]_{{\eta}_R} & C[G] \ar[d]^{{\eta}_C} \\
{R_k(T)}^W \ar[r]_{{\phi}^W_T} & {C[T]}^W}
\end{equation}
Since $T$ is a split torus, ${\phi}^W_T$ is an isomorphism.
The left vertical map is an isomorphism by \cite[Th{\'e}oreme~4]{Serre}.
This implies that the composite map ${{\eta}_C} \circ {\phi}_G$ is an
isomorphism. Thus we only need to show that ${{\eta}_C}$ is injective.

We first observe that the isomorphism $k[G]_{\ov k} 
\to {\ov k}[G_{\ov k}]$ induces a map $C[G]_{\ov k} \to 
C[G_{\ov k}]$ which is injective since the map ${C[G]}_{\ov k} \to
k[G]_{\ov k}$ is so. Now the commutative diagram
\[
\xymatrix{
C[G] \ar@{^{(}->}[r] \ar[d] & {C[G]}_{\ov k} \ar@{^{(}->}[r] &
C[G_{\ov k}] \ar[d] \\
C[T] \ar@{^{(}->}[r] & C[T_{\ov k}] \ar@{=}[r] & C[T_{\ov k}]}
\]
reduces the problem to showing that the right vertical map is injective.
Hence we can assume that $k$ is algebraically closed.

We have seen above that in this case, $C[G]$ is same as the those
functions on $G$ which are invariant under the adjoint action of
$G(k)$ on $k[G]$ given by conjugation.
 Suppose now that $f$ and $f'$ are two functions
in $C[G]$ such that they define the same elements of ${C[T]}^W$.
That is, $f$ and $f'$ define same function on the $T/W$.
We conclude from Lemma~\ref{lem:SSimple} that they take same value
on any given semi-simple element of $G$. Since the set of semi-simple
elements contains a dense open subset in $G$ ({\sl cf.} 
\cite[Theorem~11.10]{Borel}), we see that $f$ and $f'$ are same.
This proves the desired injectivity of ${{\eta}_C}$. 
\end{proof}
\begin{cor}\label{cor:charC}
Let $G$ be a connected reductive group over an algebraically closed field
$k$. Then $R_k(G)$ is naturally isomorphic to the $k$-algebra $C[G]$ of
class functions on $G$. In particular, $C[G]$ is a finite type $k$-algebra
and hence noetherian.
\end{cor}
\begin{proof} We have already seen above that in this case, $C[G]$ is
same as the algebra of class functions, i.e., those functions on $G$
which take a constant value on a conjugacy class. The corollary now
follows directly from Proposition~\ref{prop:character1}. For the last part,
we only need to show that $R_k(G)$ has the desired property. But this
follows from ~\ref{eqn:inv} and \cite[Lemma~4.7]{KV}. 
\end{proof}
\begin{cor}\label{cor:char2}
Let $G$ be a connected split reductive group of rank $r$ over a field $k$. 
Then $C[G]$ has a $k$-basis consisting of the characters of irreducible
representations of $G$. If $G$ is simply connected, then $C[G]$ is
isomorphic to the polynomial algebra $k[{\chi}_1, \cdots, {\chi}_r]$ 
in the characters of the irreducible representations with highest
fundamental weights.
\end{cor}
\begin{proof} The first part follows from Proposition~\ref{prop:character1}
since $R_k(G)$ has the desired property. The second part also follows 
from this proposition and Chevalley's theorem \cite{Che} that $R_k(G)$ has the
desired form if $G$ is simply connected.
\end{proof}
In the rest of this section, we shall assume $k$ to algebraically closed
(of any characteristic). Let $G$ be a connected linear algebraic group
and let $\Psi = C_G(g)$ be a semi-simple conjugacy class. Let ${\fm}_{\Psi}$
denote the kernel of the $k$-algebra map 
\begin{equation}\label{eqn:maximal}
R_k(G) \xrightarrow{{\chi}(g)} k.
\end{equation}
Since ${\chi}(g)$ takes the value one at the trivial character, we see that
${\fm}_{\Psi}$ is a maximal ideal of $R_k(G)$ with the residue field $k$.
We deduce some further consequences of Proposition~\ref{prop:character1}.   
\begin{cor}[\cite{ED1}]\label{cor:SSC}
Assume $G$ is connected and reductive. Then a conjugacy class $\Psi$ in $G$ 
is closed if and only if it is semi-simple.
\end{cor}
\begin{proof} It is known that the conjugacy class of a semi-simple element
is always closed ({\sl cf.} \cite[Theorem~9.2]{Borel}). The proof of the
converse follows exactly along the same line as in the proof of 
\cite[Proposition~2.4]{ED1}, once we use our Corollary~\ref{cor:charC}  
and observe that Mumford's result \cite[Chapter~1, Corollary~1.2]{GIT} 
holds over any algebraically closed field.
\end{proof}
\begin{cor}[\cite{ED1}]\label{cor:maximal1}
Let $G$ be as in Corollary~\ref{cor:SSC}. Then the correspondence
${\Psi} \mapsto {\fm}_{\Psi}$ gives a bijection between the set of
semi-simple conjugacy classes in $G$ and the maximal ideals of $R_k(G)$.
If the characteristic of $k$ is zero, the same conclusion holds for any
connected linear algebraic group.
\end{cor}
\begin{proof} The identification of the maximal ideals of $R_k(G)$ is same
as identification of the closed points of ${\rm Spec}\left(R_k(G)\right)$
and the latter is isomorphic to ${\rm Spec}\left(C[G]\right) = T/W$ by 
Proposition~\ref{prop:character1}. On the other hand, the closed points
of $T/W$ are same as its $k$-rational points which are identified
with the semi-simple conjugacy classes of $G$ by Lemma~\ref{lem:SSimple}.  
If $k$ has characteristic zero, then $G$ has a Levi subgroup $L$ and
one has $R_k(G) \cong R_k(L)$ and the argument of 
{\sl loc. cit.}, Proposition~2.5 goes through.
\end{proof}
\begin{cor}[\cite{ED1}]\label{cor:maximal2}
Let $G \inj H$ be a closed embedding of connected reductive groups and let
${\Psi} = C_G(h)$ be a semi-simple conjugacy class in $H$. Then 
${R_k(G)}_{{\fm}_{\Psi}}$ is a semi-local ring with maximal ideals
$\left \{{\fm}_{{\Psi}_1}, \cdots, {\fm}_{{\Psi}_r}\right \}$ where
${\Psi}_1 \coprod \cdots \coprod {\Psi}_r = {\Psi} \cap G$ is the disjoint
union of the semi-simple classes in $G$. If the characteristic of $k$ is
zero, then the same conclusion holds for all connected linear algebraic 
groups. 
\end{cor}
\begin{proof} The fact that ${R_k(G)}_{{\fm}_{\Psi}}$ is a semi-local ring
follows from Proposition~\ref{prop:finiteR}.
The rest of the argument is same as in the proof of {\sl loc. cit.},
Proposition~2.6 in view of Proposition~\ref{prop:character1}
and Corollary~\ref{cor:maximal1} above.
\end{proof}
\section{Preliminary Results}
For the rest of this paper, our ground field $k$ will be assumed to be 
algebraically closed and of characteristic zero. This field will be fixed
from now on. Following the approach of \cite{ED1}, we shall first prove the 
non-abelian completion theorem for the rational equivariant higher 
$K$-theory for those groups whose commutators are simply connected.  
In this section, we collect all the preliminary results needed in this
direction. 

For a finitely generated abelian group $N$ and a field $l$ of
characteristic zero, let $T = D_l(N)$ denote the split diagonalizable group 
over $l$ with the character group $N$. We shall sometime also write $T_l$ to 
emphasize the base field. Recall from \cite[Proposition~1.2]{Thom0} that
for every maximal ideal $\fm$ of the group algebra $l[N] = l[T]$, there is a 
unique diagonalizable closed subgroup $T_{\fm}$ of $T$ such that $\fm$ is the
inverse image of a maximal ideal of $l[T_{\fm}]$, and $T_{\fm}$ is the
smallest closed subgroup with this property. $T_{\fm}$ is called the 
{\sl support} of $\fm$ and its coordinate ring is the group algebra 
$l[N/{N_{\fm}}]$, where
\begin{equation}\label{eqn:max}
N_{\fm} = \left \{n \in N | 1-[n] \in {\fm}\right \},
\end{equation} 
which is same as the kernel of the map $N \xrightarrow{\phi}
{\left(l[N]/{\fm}\right)}^*$
given by $n \mapsto [n]$ modulo $\fm$.
We also recall that if $X$ is a variety over a field $l$, then for any
field extensions $l \subset l_1 \subset l_2$, there is a natural 
inclusion of sets $X(l_1) \inj X(l_2)$ of rational points of $X$. This map is  
$x \in X(l_1) \mapsto \ov{x}$, where $\ov x$ is represented by the diagonal map
${\rm Spec}(l_2) \to X {\times}_k {\rm Spec}(l_2)$ and the projection to the
first factor is ${\rm spec}(l_2) \to X$ whose image is the point $x$. 
The following lemma follows easily from Thomason's theorem in {\sl loc. cit.}.
\begin{lem}\label{lem:closure}
Let $T = D_{\Q}(N)$ be a split torus and let $t$ be a closed 
point of $T$ defined by a maximal ideal $\fm$ of ${\Q}[T]$ with residue field 
$l$. Let $L$ be an algebraically closed field containing $l$ and let 
$t_L \in T(L)$ denote the image of the point $x$ under the inclusion 
$T(l) \subset T(L)$. Then 
\[
T_{\fm} {\times}_{\Q} {\rm Spec}(L) \xrightarrow{\cong} S,
\]
where $S \subset T_L$ is the closure of the cyclic subgroup generated
by $t_L$.
\end{lem}   
\begin{proof} Let ${\fm}_L$ be the maximal ideal of $L[N]$ defining the closed
point $t_L$ of $T_L$. Since every closed subgroup of $T_L$ is diagonalizable,
we see from above that the support $T_{{\fm}_L}$ is the smallest closed 
subgroup of $T_L$ containing $t_L$ and hence is same as $S$. Thus we need to
show that $T_{\fm} {\otimes}_{\Q} L \xrightarrow{\cong} T_{{\fm}_L}$.

We see from ~\ref{eqn:max} above that $T_{\fm} = 
{\rm Spec}\left({\Q}\left[N/{N_{\fm}}\right]\right)$, where 
$N_{\fm} = {\rm Ker}\left(N \to l^*\right)$. Similarly we have
$T_{{\fm}_L} = {\rm Spec}\left(L\left[N/{N_{{\fm}_L}}\right]\right)$,
where $N_{{\fm}_L} = {\rm Ker}\left(N \to L^*\right)$. Thus it suffices to
show that $N_{\fm} = N_{{\fm}_L}$. But this follows immediately from the
commutative diagram
\begin{equation}\label{eqn:support}
\xymatrix{
& {\left({\Q}[N]\right)}^* \ar[dd] \ar[r] & l^* \ar@{^{(}->}[r] & L^*
\ar@{=}[dd] \\
N \ar[ur] \ar[dr] & & & \\
&  {\left(L[N]\right)}^* \ar[r] &  {\left(L\otimes L\right)}^* \ar[r] &
L^*.}
\end{equation}
\end{proof}
\begin{prop}\label{prop:uniqueM}
Let $G$ be a connected reductive group over $k$ and let $T$ be a maximal
torus of $G$. Let $\fm$ be a maximal ideal of $R(G)$ and let $\wt{\fm}$ be 
a maximal ideal of $R(T)$ whose inverse image in $R(G)$ is $\fm$. Assume that 
$T_{{\wt{\fm}}_k}$ is connected and is contained in the center of $G$. 
Then $\wt{\fm}$ is the unique maximal ideal of $R(T)$ which contracts to $\fm$.
\end{prop}
\begin{proof} We first remark that there is always at least one maximal
ideal of $R(T)$ whose inverse image in $R(G)$ is $\fm$, since the map
$f: {\rm Spec}\left(R(T)\right) \to {\rm Spec}\left(R(G)\right)$ is 
finite and dominant by Proposition~\ref{prop:finiteR} and hence surjective.
To prove the proposition, let $t$ be the closed point of $X =   
{\rm Spec}\left(R(T)\right)$ defined by $\wt{\fm}$ and let $x = f(t)$
be the closed point of $Y = {\rm Spec}\left(R(G)\right)$ defined by $\fm$.

Suppose $t'$ is another closed point of $X$ such that $x = f(t')$ and
consider the base change map $f_k : X_k \to Y_k$. Note that $X_k$ is same
as $T$ and $Y_k = T/W$, where $W$ is the Weyl group of $G$. 
Let $t_k$ and $t'_k$ be the images of $t$ and $t'$ respectively under the
inclusion $X(l) \subset X(k)$, where $l = {\Q}(t, t')$. It is then easy
to see that $f_k(t_k) = y_k = f(t'_k)$. It suffices to show that 
$t_k = t'_k$. However, since $k$ is algebraically closed, we can apply 
Corollary~\ref{cor:maximal2} to find that 
$f_k^{-1}(y_k) = \left\{{\fm}_{s_1}, \cdots, {\fm}_{s_r}\right\}$, where 
$C_G(t_k) \cap T = C_T(s_1) \coprod \cdots \coprod C_T(s_r)$ is the
disjoint union of the conjugacy classes in $T$. Thus it suffices to show that
$t_k \in Z(G)$. Since $S = T_{{\wt{\fm}}_k}$ is the closure of the cyclic 
subgroup of $T$ generated by $t_k$ by Lemma~\ref{lem:closure} and since
$S \subset Z(G)$ by our assumption, we only need to see that 
$Z_G(S) = Z_G(t_k)$. But this follows easily by observing that for any 
$g \in Z_G(t_k)$, the subgroup $Z_G(g)$ is closed.
\end{proof}
We fix the following notations once and for all before proceeding further.
Let $G$ be a connected reductive group over $k$ as above and let 
$T = D_k(N)$ be a maximal torus of $G$. 
Let $f : X = {\rm Spec}\left(R(T)\right) \to Y = {\rm Spec}\left(R(G)\right)$
denote the natural map of $\Q$-schemes. We fix a maximal ideal 
$\fm$ of $R(G)$ and let $\wt{\fm}$ be a maximal ideal of $R(T)$ such that
$R(G) \cap \wt{\fm} = \fm$. Let $t$ (resp. $x$) denote the closed point
of $X$ (resp. $Y$) defined by $\wt{\fm}$ (resp. $\fm$). Thus we have
$f(t) = x$. Put ${\delta} = f^{-1}(x) = \left\{t_1, \cdots , t_r\right\}$ 
with $t_1 = t$. Assume that the support $T_{\wt{\fm}}$ of 
$\wt{\fm}$ is connected. Then $S = T_{{\wt{\fm}}_k}$ is a subtorus of $T$ and 
the following property of its centralizer $Z = Z_G(S)$ in $G$ is well known.
\begin{lem}\label{lem:cent}
$Z = Z_G(S)$ is connected and reductive.
\end{lem}
\begin{proof} See \cite[Corollary~11.12 and Section~13.17, 
Corollary~2]{Borel}.
\end{proof} 
Let $W$ and $W'$ be the Weyl groups of $G$ and $Z$ respectively with respect 
to the maximal torus $T$. Then $W$ naturally acts on $T$ and hence on 
$X = {\rm Spec}\left(R(T)\right)$. 
\begin{prop}\label{prop:weyl}
For the action of the Weyl group $W$ on $X$ as above, one has $W{\fm}
= \ ({\rm orbit\ \  of} \ \ t)= {\delta}$ and the stabilizer of $t$ is 
the subgroup $W'$.
\end{prop}
\begin{proof} Since the map $f : X \to Y$ is the quotient map for the
action of the finite group $W$ on $X$ by Proposition~\ref{prop:finiteR},
$Y$ is a universal geometric quotient of $X$. In particular, we have
$f^{-1} (x) = f^{-1} \left(f(t)\right) = Wt = W{\fm}$. This proves the first
part. To prove the second part, we first observe that $Z$ is a connected
reductive group by Lemma~\ref{lem:cent} and hence we can apply 
Proposition~\ref{prop:uniqueM} to see that $t$ is the only closed point
of $X$ which lies over $f'(t)$ under the quotient map $f' : X \to
Y' = {\rm Spec}\left(R(Z)\right) = X/{W'}$. Now we can use the first part 
of the proposition with $G$ replaced by $Z$ to conclude that
$W' \subset {\rm St}_W(t)$. Thus we need to show the reverse inclusion to
finish the proof. So let $w \in {\rm St}_W(t)$. Then $w$ also stabilizes
$t_k \in T_{{\wt{\fm}}_k} = S \subset T$. Since $W = N_G(T)/T$, we can lift
$w$ to an element $\ov{w}$ of $G$ and then this lift acts on $T$ by 
conjugation.
In particular, $w \in {\rm St}_W(t_k)$ implies that $\ov{w} \in Z_G(t_k)$
and the latter is same as $Z_G(S) = Z$ as we have seen above. Hence 
$w \in N_Z(T)/T = W'$. This proves the proposition.
\end{proof}

In the next few paragraphs, we recollect some basic facts about
the fundamental groups of linear algebraic groups which are relevant to us
in this paper. Our main interest is Proposition~\ref{prop:SCC}
below which will be needed in the next section. Let $G$ be a connected
reductive group over $k$ with a fixed maximal torus $T$ and the Weyl group
$W$. Let $X^*(T)$ and $X_*(T)$ denote the groups of characters and 
cocharacters of $T$ respectively. Recall that since $T$ is a split torus, 
the group $G$ is uniquely described by its root system 
$\left(G, T, {\Phi}\right)$. The set $\Phi$ is the set of non-zero elements 
$\alpha \in X^*(T)$ such that for the action of $T$ on the Lie algebra $\fg$ 
of $G$ via the adjoint representation, the subspace 
\[
{\fg}_{\alpha} = \left\{v \in \fg | tv = {\alpha}(t)v \ \forall \
t \in T\right\}
\]
is not zero. It is clear from this that $\Phi$ is a finite set.
If we put $V = X^*(T) {\otimes}_{\Z} {\R}$, then recall that $V$ has a 
$W$-invariant inner product $<,>$ and this makes $(V, \Phi)$ an abstract
root system. If ${\Phi}^{\vee} \subset V^*$ is set of all elements
${\alpha}^{\vee}$ such that ${\alpha}^{\vee}(v) = 
\left <\frac{2}{<\alpha, \alpha>} \alpha, v \right>$ for $\alpha \in \Phi$,
then $(V^*, {\Phi}^{\vee})$ is called the dual root system.
Furthermore, there is a perfect pairing $X^*(T) \times X_*(T)
\to {\Z}$ given by $(f, g) \mapsto <f, g> = f \circ g \in X_*({\G}_m)
= \Z$, and this identifies $X_*(T) {\otimes}_{\Z} {\R}$ with $V^*$ and the
integrality condition of the root system implies that the subset
\[
\Lambda \left({\Phi}^{\vee}\right) = \left\{v \in V^* |
\left<{\alpha}^{\vee}, v\right> \in \Z\right\}
\]
is a lattice in $V^*$ and is in fact contained in $X_*(T)$.

The inclusion $T \inj G$ defines a natural map
$X_*(T) = {\pi}_1(T) \to {\pi}_1(G)$ which induces an isomorphism
({\sl cf.} \cite[Section~31.8]{Hump})
\begin{equation}\label{eqn:Pi1}
\frac{X_*(T)}{\Lambda \left({{\Phi}'}^{\vee}\right)} \xrightarrow{\cong}
{\pi}_1(G),
\end{equation}
where $\left(G', T', {\Phi}'\right)$ is the root system of the derived
subgroup $G'= (G, G)$ of $G$. It is easy to see that in case $G$ is a
complex reductive group, then ${\pi}_1(G)$ coincides with the topological
fundamental group of $G$. The second part of the following result was proved in
\cite[Lemma~2.5]{ED3} for the complex algebraic groups by purely topological
means.
\begin{lem}\label{lem:SC1}
Let $G$ be as above. Then ${\pi}_1(G)$ is a finitely generated abelian group
and it is finite if and only if $G$ is semi-simple. Furthermore, the derived 
subgroup $G'$ of $G$ is simply connected if and only if ${\pi}_1(G)$ is 
torsion-free.
\end{lem}
\begin{proof} The first part of the lemma follows directly from 
~\ref{eqn:Pi1} above once we observe that the rank of 
$\Lambda \left({{\Phi}'}^{\vee}\right)$ is same as the semi-simple rank of $G$.
To prove the second part, let $T' = G' \cap T$. Then it is well known that
$T'$ is a maximal torus of $G'$ and $T/{T'}$ is also a torus. In particular,
$\frac{X_*(T)}{X_*(T')} \cong X_*(T/{T'})$ is a free abelian group. The lemma
now follows immediately from the short exact sequence
\begin{equation}\label{eqn:SC2}
0 \to \frac{X_*(T')}{\Lambda \left({{\Phi}'}^{\vee}\right)} \to
\frac{X_*(T)}{\Lambda \left({{\Phi}'}^{\vee}\right)} \to 
\frac{X_*(T)}{X_*(T')} \to 0.
\end{equation}
\end{proof}
The following result about the fundamental groups of the centralizers of
the sub-tori of $G$ will play a crucial role in the next section to prove
the non-abelian completion theorem in a special case.
\begin{prop}\label{prop:SCC}
Let $G$ be as in Lemma~\ref{lem:SC1} and assume that the derived subgroup
$G'$ of $G$ is simply connected. Then for any torus $S \subset G$ in $G$,
the derived subgroup $Z'$ of the centralizer $Z = Z_G(S)$ is also simply
connected. In particular, the commutators of the Levi subgroups of
any parabolic subgroup of $G$ are simply connected. 
\end{prop}
\begin{proof} Using Lemma~\ref{lem:SC1}, it is enough to show that
${\pi}_1(Z)$ is torsion-free.
Let $T$ be a maximal torus of $G$ containing $S$ and let
$\left(G, T, {\Phi}\right)$ be the associated root system. It is then
well known ({\sl cf.} \cite[Section~11.17]{Borel}) that there is a subset 
$I$ of $\Phi$ such that $S$ is the identity component of the subgroup 
$\stackrel{} {\underset {\alpha \in I}{\bigcap}} {\rm Ker}(\alpha)$ of $T$
and hence $Z =  Z_G(S)$ is the Levi subgroup of a parabolic subgroup of
$P$ containing $T$ ({\sl cf.}\cite[Proposition~14.18]{Borel}).
It is easy to check in this case that $Z' = Z \cap G'$. 
Let ${\Phi}'_I \subset {\Phi}'$ be the subset consisting of the linear 
combination of elements of $I$. Then $X_*(S)$ has an orthogonal
complement ${\left(X_*(S)\right)}^{\perp}$ inside $X_*(T)$ with respect
to the inner product as described above and so has 
${\Lambda \left({{\Phi}'_I}^{\vee}\right)}$ inside 
${\Lambda \left({{\Phi}'}^{\vee}\right)}$. In particular, the quotient
${\pi}_1(Z) = \frac{X_*(T)}{\Lambda \left({{\Phi}'_I}^{\vee}\right)}$ is
a direct summand of $\frac{X_*(T)}{\Lambda \left({{\Phi}'}^{\vee}\right)}
= {\pi}_1(G)$ and the latter group is torsion-free.
The second part of the proposition follows from the first because
the Levi subgroup of any parabolic subgroup $P$ of $G$ is isomorphic to
the centralizer of the subtorus $S$ of $T$ as above ({\sl cf. loc. cit.},
Corollary~14.19).
\end{proof}  

We end this section with the following reduction step
which can be used to prove the main results of this paper for non-reductive
groups.
\begin{prop}\label{prop:NR}
Let $G$ be a connected linear algebraic group and let $L$ be a Levi subgroup
of $G$. Suppose $G$ acts on a smooth quasi-projective variety $X$. Then
the restriction map $K^G_*(X) \to K^L_*(X)$ is an isomorphism. In
particular, $R(G) \xrightarrow{\cong} R(L)$.
\end{prop}
\begin{proof} Let $U = R_uG$ denote the unipotent radical of $G$.
Morita isomorphism ({\sl cf.} \cite[Theorem~1.10]{Thom1}) implies that 
there is an isomorphism $K^G_*\left(G, G  \stackrel{L}
{\underset {}{\times}} X\right) \xrightarrow{\cong} K^L_*(X)$. Since $X$ is a 
$G$-space, there is a $G$-equivariant isomorphism 
$G  \stackrel{L} {\underset {}{\times}} X \cong G/L \times X$. Thus we get an 
isomorphism 
\begin{equation}\label{eqn:Morita}
K^L_*(X) \cong K^G_*\left(G, G/L \times X\right)
\cong K^G_*\left(U \times X\right).
\end{equation}
Now, as $U$ is a split unipotent group, there is a composition series
of closed subgroups of $G$ such that each successive quotient is isomorphic
to the additive group ${\G}_a$. Using a descending induction and homotopy
invariance, one gets an isomorphism $K^G_*\left(U \times X\right) \cong
K^G_*(X)$.
\end{proof}
\section{Non-abelian Completion Theorem}
Let $G$ be a connected reductive group over $k$ with a maximal torus $T$.
Let $\fm$ be a maximal ideal of $R(G)$ and let $\wt{\fm}$ be a maximal
ideal of $R(T)$ whose inverse image in $R(G)$ is $\fm$. Let 
$S = T_{{\wt{\fm}}_k} \subset T$ be the support of $\wt{\fm}$ and let
$Z = Z_G(S)$ denote the centralizer of $S$ in $G$. Put $\ov{\fm} =
\wt{\fm} \cap R(Z)$. Let $X$ be a smooth quasi-projective variety with
a $G$-action and let $f: X^{\wt{\fm}} \inj X$ be the fixed locus for
the action of $S$ on $X$. It is known that ({\sl cf.} 
\cite[Proposition~3.1]{Thom0}) that $X^{\wt{\fm}}$ is smooth. 
Let $d$ be the codimension of $X^{\wt{\fm}}$ in $X$.
Let $\sN$ denote the conormal bundle of $X^{\wt{\fm}}$ in $X$ and 
put ${\lambda}_{-1}(\sN) = \left[{\sO}_{X^{\wt{\fm}}}\right] - 
\left[{\Lambda}^1 \sN \right] + \cdots \underline + 
\left[{\Lambda}^d \sN \right]$
as a class in $K^G_0\left(X^{\wt{\fm}}\right)$.
Recall from ~\ref{eqn:res} that there are
restriction maps of ${\wh{R(G)}}_{\fm}$-modules
\begin{equation}\label{eqn:res1}
{res}_{\fm} : {\wh{K^G_i(X)}}_{\fm} \to {\wh{K^Z_i(X)}}_{\fm} \to
{\wh{K^Z_i(X)}}_{\ov{\fm}}, \ \ {\rm and}
\end{equation}  
\begin{equation}\label{eqn:res2}
f^*: {\wh{K^Z_i(X)}}_{\ov{\fm}} \to {\wh{K^Z_i(X^{\wt{\fm}})}}_{\ov{\fm}}.
\end{equation}
We begin with the following extension of the localization result
\cite[Lemme~3.2]{Thom0} for the action of diagonalizable groups to the
case of arbitrary linear algebraic groups.
\begin{lem}\label{lem:Local}
Let $G$ be a connected reductive group acting on a smooth quasi-projective
variety $X$ as above. Let $\wt{\fm}$ be a maximal ideal of $R(T)$ with
$\fm = \wt{\fm} \cap R(G)$. Assume that the support $T_{\wt{\fm}}$ is
connected and is contained in the center of $G$. Then the pullback map 
\begin{equation}\label{eqn:Local1}
f^* : {K^G_i(X)}_{\fm} \to {K^G_i(X^{\wt{\fm}})}_{\fm} 
\end{equation}
is an isomorphism. In particular, the map
\begin{equation}\label{eqn:Local11}
f^* : {K^G_i(X)} {\otimes}_{R(G)} {\wh{R(G)}}_{\fm} \to
{K^G_i(X^{\wt{\fm}})} {\otimes}_{R(G)} {\wh{R(G)}}_{\fm}
\end{equation}
is also an isomorphism. 
\end{lem}
\begin{proof} Since $X$ and $X^{\wt{\fm}}$ are smooth, we have the push
forward map $f_* : K^G_i(X^{\wt{\fm}}) \to K^G_i(X)$ such that the 
composite map $f^* \circ f_*$ is multiplication with $\alpha
= {\lambda}_{-1}(\sN)$ by the self-intersection formula 
({\sl cf.} \cite[Theorem~2.1]{VV}). The map $f_*$ is an isomorphism after 
localization at $\fm$ by \cite[Th{\'e}oreme~2.2]{Thom0}. Hence it suffices
to show that $\alpha$ acts as a unit on ${K^G_i(X^{\wt{\fm}})}_{\fm}$.

We have seen in the proof of Proposition~\ref{prop:NR} that there are 
natural maps of $R(G)$-modules 
\begin{equation}\label{eqn:split}
K^G_i\left(X^{\wt{\fm}}\right) \xrightarrow{p^{!}} 
K^T_i\left(X^{\wt{\fm}}\right) \xrightarrow{\cong}
K^G_*\left(G, G/T \times X^{\wt{\fm}}\right) \xrightarrow{p_{!}}
K^G_i\left(X^{\wt{\fm}}\right)
\end{equation}
such that the composite is the identity map 
({\sl cf.} \cite[Section~1.6]{Thom2}). Hence it suffices to show that
$\alpha$ acts as a unit on ${K^T_i\left(X^{\wt{\fm}}\right)}_{\fm}$. Since 
$S = {T_{\wt{\fm}}}_k \subset Z(G)$, we can apply 
Proposition~\ref{prop:uniqueM} to conclude that $\wt{\fm}$ is the unique 
maximal ideal of $R(T)$ which contracts to $\fm$ in $R(G)$. In particular, we 
have ${K^T_i\left(X^{\wt{\fm}}\right)}_{\fm} \cong 
{K^T_i\left(X^{\wt{\fm}}\right)}_{\wt{\fm}}$. Thus we need to show that
$\beta = p^{!}(\alpha)$ acts as a unit on
${K^T_i\left(X^{\wt{\fm}}\right)}_{\wt{\fm}}$. But this follows from
\cite[Lemme~3.2]{Thom0} since $\beta$ is the class of ${\lambda}_{-1}(\sN)$
in $K^T_0\left(X^{\wt{\fm}}\right)$.

Finally, the isomorphism of ~\ref{eqn:Local11} is proved by observing
that the left side of ~\ref{eqn:Local1} is same as 
${K^G_i(X)} {\otimes}_{R(G)} {R(G)}_{\fm}$ ({\sl cf.} 
\cite[Theorem~4.4]{Matsu}) and similarly the right side. Moreover,
~\ref{eqn:Local1} is an isomorphism of ${R(G)}_{\fm}$-modules and hence
remains an isomorphism when tensored with the ${R(G)}_{\fm}$-module
${\wh{R(G)}}_{\fm}$.
\end{proof}     
\begin{prop}\label{prop:NACS}
Let $G$ be a connected reductive group acting on a smooth quasi-projective
variety $X$ as above. Assume that the derived subgroup $G'$ of $G$ is
simply connected. Let $\wt{\fm}$ be a maximal ideal of $R(T)$ with
$\fm = \wt{\fm} \cap R(G)$. Assume that the support $T_{\wt{\fm}}$ is 
connected. Then the map 
\[
{\wh{K^G_i(X)}}_{\fm} \xrightarrow{{res}_{\fm}}
{\wh{K^Z_i(X)}}_{\ov{\fm}}
\]
is an isomorphism of $\wh{R(G)}_{\fm}$-modules.
\end{prop}
\begin{proof} Let $W$ and $W'$ be the Weyl groups of $G$ and $Z$ respectively.
We have the maps of $\Q$-algebras
\begin{equation}\label{eqn:map}
\xymatrix{
R(G) \ar[r]^{f^G_Z} \ar@/_1pc/[rr]_{f^G_T} & R(Z) \ar[r]^{f^Z_T} & R(T)}
\end{equation}
which are all finite by Proposition~\ref{prop:finiteR}.
Put $\ov{\fm} = \wt{\fm} \cap R(Z)$ and let ${f^G_Z}^{-1}(\fm) = \left\{
{\fm}_1, \cdots , {\fm}_s \right\}$ with ${\fm}_1 = \ov{\fm}$ and 
${f^Z_T}^{-1}({\fm}_i) = \left\{{\fm}_{i1}, \cdots, {\fm}_{is_i}\right\}$
with ${\fm}_{11} = \wt{\fm}$. Since ${f^Z_T}^{-1}({\fm}_1) = \wt{\fm} =
{\fm}_{11}$ by Proposition~\ref{prop:uniqueM}, we can write
\[
{f^G_T}^{-1}(\fm) = \{\wt{\fm}\} \bigcup 
\left(\stackrel{} {\underset {i \ge 2}{\bigcup}}
\stackrel{s_i} {\underset {j=1}{\bigcup}} \{ {\fm}_{ij}\}\right).
\]
In particular, the natural maps
\begin{equation}\label{eqn:map1}
{\wh{K^T_i(X)}}_{\fm} \longrightarrow \stackrel{} 
{\underset {i \ge 2, j \ge 1}{\prod}}
{\wh{K^T_i(X)}}_{{\fm}_{ij}} \bigoplus {\wh{K^T_i(X)}}_{\wt{\fm}}
\end{equation}
\begin{equation}\label{eqn:map2}
{\wh{K^T_i(X)}}_{\ov{\fm}} \longrightarrow {\wh{K^T_i(X)}}_{\wt{\fm}}
\end{equation}
are isomorphisms by \cite[Lemma~3.1]{ED3}.

Since $S = {T_{\wt{\fm}}}_k$ is connected and $G'$ is
simply connected, we see from Lemma~\ref{lem:cent} and 
Proposition~\ref{prop:SCC} that $Z$ is a connected reductive subgroup of $G$
of the same rank with simply connected commutator subgroup.
Hence the natural maps 
\begin{equation}\label{eqn:Merk}
K^G_i(X) {\otimes}_{R(G)} R(T)\xrightarrow{{\theta}^G_T} K^T_i(X) , \ \
K^Z_i(X) {\otimes}_{R(Z)} R(T) \xrightarrow{{\theta}^Z_T} K^T_i(X)
\end{equation}
isomorphisms by \cite[Proposition~8]{Merkur}.   
Taking the invariants for the Weyl groups both sides, we get isomorphisms
\begin{equation}\label{eqn:Merk1}
K^G_i(X) \xrightarrow{\cong} {\left(K^T_i(X)\right)}^W \ {\rm and} \
K^Z_i(X) \xrightarrow{\cong} {\left(K^T_i(X)\right)}^{W'}. 
\end{equation}
Taking the $\fm$-adic completions and using \cite[Lemma~3.2]{ED3}, we get 
isomorphisms of ${\wh{R(G)}}_{\fm}$-modules
\begin{equation}\label{eqn:Merk1}
{\wh{K^G_i(X)}}_{\fm} \xrightarrow{\cong} 
{\left({\wh{K^T_i(X)}}_{\fm}\right)}^W \ {\rm and} \
{\wh{K^Z_i(X)}}_{\ov{\fm}} \xrightarrow{\cong} 
{\left({\wh{K^T_i(X)}}_{\ov{\fm}}\right)}^{W'}. 
\end{equation}

For the action of $W$ on ${\wh{K^T_i(X)}}_{\fm}$ as in ~\ref{eqn:map1}
and for $i,j \ge 1$, put 
\[
W_{ij} = \left\{ w \in W | w {\wh{K^T_i(X)}}_{{\fm}_{ij}} =
{\wh{K^T_i(X)}}_{{\fm}_{ij}}\right\}.
\]
Then Proposition~\ref{prop:weyl} implies that $W_{11} = W'$ and 
${f^G_T}^{-1}(\fm)$ is the orbit of $\wt{\fm}$ for the action of $W$.
Using the fact that ${\fm}_{11} = \wt{\fm}$, we conclude using
\cite[Lemma~3.3]{ED3} that the natural map
\begin{equation}\label{eqn:Merk2}
{\left({\wh{K^T_i(X)}}_{\wt{\fm}}\right)}^{W'} \to
{\left({\wh{K^G_i(X)}}_{\fm}\right)}^{W}
\end{equation}
is an isomorphism. Moreover, ~\ref{eqn:map2} and ~\ref{eqn:Merk1} show
that the term on the left is same as ${\wh{K^Z_i(X)}}_{\ov{\fm}}$.
We conclude using ~\ref{eqn:Merk1} again that the map 
${\wh{K^G_i(X)}}_{\fm} \to {\wh{K^Z_i(X)}}_{\ov{\fm}}$ is an isomorphism of
${\wh{R(G)}}_{\fm}$-modules.
\end{proof}

Our final goal in this section is to deduce Theorem~\ref{thm:NAL} from
Proposition~\ref{prop:NACS} using the change of groups argument of \cite{ED3}.
Accordingly, we first establish the following steps necessary for the purpose.
Let $G$ be a connected and reductive group over $k$ with a fixed maximal
torus $T$.
\begin{lem}\label{lem:incl}
For a given maximal ideal $\fm$ of $R(G)$, we can embed $G$ as a closed 
subgroup of a group $H$ such that $H$ is product of general linear groups and 
$\fm$ is the only maximal ideal lying over $\fm \cap R(H)$ under the natural 
map $R(H) \to R(G)$.
\end{lem}
\begin{proof} Put $X = {\rm Spec}\left(R(G)\right)$ and let $x$ denote the
closed point of $X$ defined by the maximal ideal $\fm$. Let $L$ denote
the residue field of $x$ and let $x_k$ be the image of $x$ under the 
inclusion $X(L) \subset X(k)$.  Then by Corollaries~\ref{cor:charC} and 
~\ref{cor:maximal1}, we have $R_k(G) \xrightarrow{\cong} C[G]$ and $x_k$
is given by a unique conjugacy class ${\fm}_{\Psi}$ of a semi-simple 
element $h$ of $G$. Hence we can use \cite[Proposition~2.8]{ED1}
to embed $G \inj H = \stackrel{} {\underset {i}{\prod}}GL_{n_i}$ as a closed
subgroup such that ${\Psi}_{H} (h) \cap G = {\Psi}_{G} (h)$.
Put $Y = {\rm Spec}\left(R(H)\right)$ and let $f: X \to Y$ be the induced
map, which is finite by Proposition~\ref{prop:finiteR}. Putting $y =
f(x)$, we see by Corollary~\ref{cor:maximal2} that $x_k$ is the only closed 
point lying over $y_k = f(x_k)$ under the finite map 
$X_k \xrightarrow{f_k} Y_k$. We conclude that $x$ is the only closed point
of $X$ lying over $y$.
\end{proof}

Let us fix a maximal ideal $\fm$ of $R(G)$ and  let $G \inj H$ be an 
embedding as a closed subgroup as in Lemma~\ref{lem:incl}.
Let $\wt{\fm}$ be a maximal ideal of $R(T)$ lying over $\fm$. Such an
ideal always exists as follows from Proposition~\ref{prop:finiteR}.
Let $T_{\wt{\fm}}$ be the support of $\wt{\fm}$ and let $Z$ denote the
centralizer of $S = T_{\wt{\fm}} {\otimes}_{\Q} k$ in $G$. Let $T'$ be a 
maximal torus of $H$ containing $T$ and $Z'$ the centralizer of $S$ in $H$. 
It is then clear that $Z = Z' \cap G$. Put
\[
\ov{\fm} = {\wt{\fm}} \cap R(Z), \  {\fm}' = {\fm} \cap R(H), 
\ {\wt{\fm}}' = {\wt{\fm}} \cap R(T'), \
{\rm and} \ {\ov{\fm}}' = {\ov{\fm}} \cap R(Z').
\]   
Note that the map $R(T') \to R(T)$ is surjective and ${\wt{\fm}}'$ is the
unique maximal ideal of $R(T')$ lying over $\wt{\fm}$.
Let $X$ be a smooth quasi-projective variety over $k$ with an action of $G$
and let $f : X^{\wt{\fm}} \inj X$ be the inclusion of the fixed locus for the
action of $S$ on $X$. Put $X' =  H \stackrel{G} {\underset {}{\times}} X$.
Then $H$ acts naturally on $X'$. Let $f': {X'}^{{\wt{\fm}}'} \inj
X'$ denote the inclusion of the fixed locus for the action of 
$S \subset T'$ on $X'$. Then $X'$ and ${X'}^{{\wt{\fm}}'}$ are also smooth
and $f'$ is a regular embedding. Since $S$ is contained in the center of $Z$,
we see that $X^{\wt{\fm}}$ is $Z$-invariant and similarly 
${X'}^{{\wt{\fm}}'}$ is $Z'$-invariant. Thus there is a natural map
$Z' \times  X^{\wt{\fm}} \to {X'}^{{\wt{\fm}}'}$ which in fact descends to a 
map ${\phi}: Z' \stackrel{Z} {\underset {}{\times}} X^{\wt{\fm}} \to
{X'}^{{\wt{\fm}}'}$.
\begin{lem}\label{lem:quotient}
The map $\phi$ is an isomorphism of $Z'$-spaces.
\end{lem}   
\begin{proof} We first observe from \cite[Proposition~8.18]{Borel} that
there is a dense open subset $U_1$ of $S$ such that $Z_H(S) = Z_H(h)$
and hence $Z_G(S) = Z_G(h)$ for all $h \in U_1$. Next we claim that there
is a dense open subset $U_2$ of $S$ such that $X^S = X^h$ and ${X'}^S =
{X'}^h$ for all $h \in U_2$.

Since $S$ is a torus and $X$ (and hence $X'$) has ample line bundles,
there is a finite cover of $X$ and $X'$ by $S$-invariant affine open
sets ({\sl cf.} \cite[Corollary~3.11]{Sumihiro}). Thus by choosing our
desired open set as the finite intersection of open subsets for the 
element of the affine cover, we can reduce to the case when $X$ and 
$X'$ are affine. Then there is a linear representation $V$ of
$S$ such that $X \inj V$ is an $S$-equivariant closed embedding. Since
$X^S = V^S \cap X$ and $X^h = V^h \cap X$, we can assume that $X = V$.
Then we can write $V = V_0 \oplus V_1 \oplus \cdots \oplus V_r$ as a
direct sum of $S$-invariant subspaces, where $S$ acts on $V_i$ by a 
character ${\chi}_i$ and ${\chi}_i \neq {\chi}_j$ for $i \neq j$.
We can find a dense open subset $U_2 \subset S$ such that ${\chi}_i(h) 
\neq {\chi}_j(h)$ for $i \neq j$ and $h \in U_2$. For such $h$, it is
easy to see that $V_0 = V^S = V^h$. The same proof works also for
$X'$. This proves our claim.

Taking $U = U_1 \cap U_2$, we see that $U$ is a dense open subset of $S$
such that for all $h \in U$, one has $Z_H(S) = Z_H(h)$, $X^S = X^h$
and ${X'}^S = {X'}^h$. In particular, we get
\[
Z' \stackrel{Z} {\underset {}{\times}} X^{\wt{\fm}} \cong
Z_{H}(h) \stackrel{Z_G(h)} {\underset {}{\times}} X^h \ {\rm and} \
{X'}^{{\wt{\fm}}'} \cong {X'}^h.
\]
The lemma now follows from this and \cite[Lemma~3.5]{ED1}.
\end{proof}
\begin{remk}\label{remk:quotient1} Although the above lemma has been stated
for the maximal ideals of $R(G)$, it is easy to see that the proof works
for all prime ideals as long as their supports are connected.
\end{remk}
\begin{prop}\label{prop:NAL*0}
Let the notations be as above and consider the commutative diagram
\begin{equation}\label{eqn:NAL*1}
\xymatrix{
{\wh{K^H_i(X')}}_{{\fm}'} \ar[r] \ar@/^1pc/[rr]^{{f'}^!_{{\fm}'}} \ar[d] &
{\wh{K^{Z'}_i(X')}}_{{\fm}'} \ar[d] \ar[r]_{{f'}^*} & 
{\wh{K^{Z'}_i\left(X'^{{\wt{\fm}}'}\right)}}_{{\ov{\fm}}'} \ar[d] \\
{\wh{K^G_i(X)}}_{\fm} \ar[r] \ar@/_1pc/[rr]_{{f}^!_{\fm}} &
{\wh{K^Z_i(X)}}_{\fm} \ar[r]^{f^*} &
{\wh{K^Z_i\left(X^{\wt{\fm}}\right)}}_{\ov{\fm}}}
\end{equation}
The map $f^!_{\fm}$ is an isomorphism if ${f'}^!_{{\fm}'}$ is so.
\end{prop}
\begin{proof} We consider the following commutative diagram.
\begin{equation}\label{eqn:NAL*2}
\xymatrix{
{\wh{K^H_i(X')}}_{{\fm}'} \ar[d]_{{f'}^!_{{\fm}'}} \ar[r] &
{\wh{K^G_i(X)}}_{{\fm}'} \ar[r] \ar[d] &
{\wh{K^G_i(X)}}_{\fm} \ar[d]^{{f}^!_{\fm}} \\
{\wh{K^{Z'}_i\left(X'^{{\wt{\fm}}'}\right)}}_{{\ov{\fm}}'} \ar[r] &
{\wh{K^Z_i\left(X^{\wt{\fm}}\right)}}_{{\ov{\fm}}'} \ar[r] &
{\wh{K^Z_i\left(X^{\wt{\fm}}\right)}}_{\ov{\fm}}}
\end{equation}
The first horizontal arrow in the top row is isomorphism by the Morita
isomorphism and the second one is an isomorphism by 
Proposition~\ref{prop:finiteR}, Lemma~\ref{lem:incl}
and \cite[Lemma~3.1]{ED3}.
The left vertical arrow is an isomorphism by our assumption.
The first horizontal arrow in the bottom row is an isomorphism by
Lemma~\ref{lem:quotient} and the Morita isomorphism.
We conclude that the middle vertical arrow is an isomorphism.
We now claim that $\ov{\fm}$ is the only maximal ideal of $R(Z)$ which 
contracts to ${\ov{\fm}}'$ in $R(Z')$. 

To see this, suppose ${\ov{\fm}}_1$
is another such maximal ideal of $R(Z)$. Since the map $R(Z) \to R(T)$
is finite and dominant by Proposition~\ref{prop:finiteR}, we can choose a
maximal ideal ${\wt{\fm}}_1$ of $R(T)$ which contracts to ${\ov{\fm}}_1$
in $R(Z)$. Since the map $R(T') \to R(T)$ is surjective, 
${\wt{\fm}}_1$ is the unique maximal ideal of $R(T)$ lying over
${\wt{\fm}}'_1 = {\wt{\fm}}_1 \cap R(T')$ of $R(T')$.
Hence we get two distinct maximal ideals ${\wt{\fm}}'$ and
${\wt{\fm}}'_1$ of $R(T')$ whose intersection with $R(Z')$ is same
as ${\wt{\fm}}' \cap R(Z') = {\ov{\fm}}'$, which is a contradiction by
Proposition~\ref{prop:uniqueM} as $S$ is contained in the center of $Z'$.
This proves the claim.   
Using the claim and {\sl loc. cit.}, Lemma~3.1, we see that the second
horizontal arrow on the bottom row in ~\ref{eqn:NAL*2} is an isomorphism. 
A diagram chase shows that the map $f^!_{\fm}$ is also an isomorphism.
\end{proof}

{\bf{Proof of Theorem~\ref{thm:NAL}:}}
We first show the isomorphism of the horizontal arrows in ~\ref{eqn:NAL*}.
The isomorphism of ${\wt{f}}^*$ and $f^*$ follows from 
Lemma~\ref{lem:Local}. Thus it suffices to show that the
composite horizontal maps in the top and the bottom rows are isomorphisms.
We embed $G$ as a closed subgroup of $H$ as in Lemma~\ref{lem:incl}.
We first show that the composite map $f^!_{\fm}$ in the bottom row of
~\ref{eqn:NAL*} is an isomorphism. In view of Proposition~\ref{prop:NAL*0},
it suffices to show that ${f'}^!_{{\fm}'}$ is an isomorphism. Hence we can
assume that the derived subgroup $G'$ of $G$ is simply connected.
In this case, the conclusion follows from Lemma~\ref{lem:Local}
and Proposition~\ref{prop:NACS}.

Now we prove the isomorphism of the composite map ${\wt{f^!}}_{\fm}$ in the
top row of ~\ref{eqn:NAL*}.
We first assume that the derived subgroup of $G$ is simply connected.
This implies from Lemma~\ref{lem:cent} and Proposition~\ref{prop:SCC}
that $Z$ also has the similar property. Hence the natural map
\begin{equation}\label{eqn:GZ}
K^G_i(X) {\otimes}_{R(G)} R(Z) \to K^Z_i(X)
\end{equation}
is an isomorphism by \cite[Proposition~2.1]{ED3}.
Next we take $X = {\rm Spec}(k)$ in the bottom row of ~\ref{eqn:NAL*}
to see that map ${\wh{R(G)}}_{\fm} \to {\wh{R(Z)}}_{\ov{\fm}}$ is an
isomorphism of complete local rings. Tensoring this isomorphism with
$K^G_i(X)$, we get 
\[
\xymatrix@C.9pc{
{K^G_i(X) {\otimes}_{R(G)} {\wh{R(G)}}_{\fm}} \ar[r]^{\cong} \ar[rrd] &
{K^G_i(X) {\otimes}_{R(G)} {\wh{R(Z)}}_{\ov{\fm}}} \ar@{=}[r] &
{\left(K^G_i(X) {\otimes}_{R(G)} R(Z)\right) {\otimes}_{R(Z)} 
{\wh{R(Z)}}_{\ov{\fm}}} \ar[d]^{\cong} \\
& &  {K^Z_i(X) {\otimes}_{R(Z)} {\wh{R(Z)}}_{\ov{\fm}}},}
\] 
where the vertical arrow is isomorphism by ~\ref{eqn:GZ}. The proof in the
simply connected case is now completed by Lemma~\ref{lem:Local}.

We now prove the isomorphism of ${\wt{f}}^!_{\fm}$ in ~\ref{eqn:NAL*} in the
general case. So let $G$ be a connected and reductive group with a fixed
maximal torus $T$. Let $\fm$ be the given maximal ideal of $R(G)$.
We embed $G$ as a closed subgroup of $H$ as in Lemma~\ref{lem:incl}
and follow the same notations as above. We have seen in the proof of
Proposition~\ref{prop:NAL*0} that in this case, $\ov{\fm}$ is the only
maximal ideal of $R(Z)$ whose intersection with $R(Z')$ is
${\ov{\fm}}'$. Since the maps of these representation rings are finite
by Proposition~\ref{prop:finiteR}, we see that ${\wh{R(G)}}_{{\fm}'}
\cong {\wh{R(G)}}_{\fm}$ and ${\wh{R(Z)}}_{{\ov{\fm}}'} \cong
{\wh{R(Z)}}_{{\ov{\fm}}}$. We consider the following commutative diagram.
\begin{equation}\label{eqn:final}
\xymatrix{
K^H_i(X') {\otimes}_{R(H)} {\wh{R(H)}}_{{\fm}'} \ar[r]^{{\wt{f'}}^!_{{\fm}'}}
\ar[d]^{\cong}_{l_1} &
K^{Z'}_i({X'}^{{{\wt{\fm}}'}}) {\otimes}_{R(Z')} {\wh{R(Z')}}_{{\ov{\fm}}'}  
\ar[d]_{\cong}^{r_1} \\  
K^G_i(X) {\otimes}_{R(H)} {\wh{R(H)}}_{{\fm}'} \ar[d]^{\cong}_{l_2} & 
K^{Z}_i({X}^{{\wt{\fm}}}) {\otimes}_{R(Z')} {\wh{R(Z')}}_{{\ov{\fm}}'}  
\ar[d]_{\cong}^{r_2} \\  
K^G_i(X) {\otimes}_{R(G)} \left(R(G) {\otimes}_{R(H)}  
{\wh{R(H)}}_{{\fm}'}\right) \ar[d]^{\cong}_{l_3} &
K^{Z}_i({X}^{{\wt{\fm}}})  {\otimes}_{R(Z)} \left(R(Z) {\otimes}_{R(Z')}  
{\wh{R(Z')}}_{{\ov{\fm}}'}\right) \ar[d]_{\cong}^{r_3} \\
K^G_i(X) {\otimes}_{R(G)} {\wh{R(G)}}_{{\fm}'} \ar[d]^{\cong}_{l_4} &
K^{Z}_i({X}^{{\wt{\fm}}})  {\otimes}_{R(Z)} 
{\wh{R(Z)}}_{{\ov{\fm}}'} \ar[d]_{\cong}^{r_4} \\
K^G_i(X) {\otimes}_{R(G)} {\wh{R(G)}}_{\fm} \ar[r]^{{\wt{f}}^!_{\fm}}  &
K^{Z}_i({X}^{{\wt{\fm}}})  {\otimes}_{R(Z)} 
{\wh{R(Z)}}_{{\ov{\fm}}}}
\end{equation}
Since the derived subgroup of $H$ is simply connected by our choice of $H$,
the isomorphism of the top horizontal arrow in ~\ref{eqn:final} is 
already shown. The arrows $l_1$ and $r_1$ are the Morita isomorphisms.
The isomorphism of $l_2$ and $r_2$ is obvious. Next we see from 
~\ref{eqn:inv} that $R(G)$ is the ring of invariants for the Weyl group 
action on $R(T)$ and hence is noetherian ({\sl cf.} \cite[Lemma~4.7]{KV}). 
The isomorphism of $l_3$ and $r_3$ is now an immediate consequence 
Proposition~\ref{prop:finiteR} and \cite[Theorem~8.7]{Matsu}. 
Finally, we have already concluded above that the arrows $l_4$ and $r_4$ are
isomorphisms. A diagram chase shows that the bottom horizontal arrow
is an isomorphism. This completes the proof of the isomorphism of the
horizontal arrows in ~\ref{eqn:NAL*}. 

Since we have just observed that $R(G)$ is noetherian,  the
injectivity of the vertical arrows in ~\ref{eqn:NAL*}  follows from
the following elementary lemma.
$\hfill \square$
\\
\begin{lem}\label{lem:inj*}
Let $A$ be a commutative noetherian ring and let $\wh{A}$ denote the
completion of $A$ with respect to a given ideal $I$ in $A$.
Then for any $A$-module $M$, the natural map $M {\otimes}_A {\wh{A}} \to
\wh{M}$ is injective.
\end{lem}
\begin{proof} Since $A$ is noetherian, the lemma is obvious if $M$ is
a finitely generated $A$-module ({\sl cf.} \cite[Theorem~8.7]{Matsu}). 
In general, we can find a direct system 
${\left\{M_{\gamma}\right\}}_{\gamma \in \Gamma}$ of finitely generated
submodules of $M$ such that $\stackrel{} 
{\underset {\gamma}{\varinjlim}} M_{\gamma} \xrightarrow{\cong} M$.
In particular, we get $\stackrel{}{\underset {\gamma}{\varinjlim}}   
M_{\gamma} {\otimes}_A {\wh{A}} \xrightarrow{\cong} M {\otimes}_A {\wh{A}}$.
Now we consider the following commutative diagram.
\[
\xymatrix{
{\stackrel{}{\underset {\gamma}{\varinjlim}}   
M_{\gamma} {\otimes}_A {\wh{A}}} \ar[r]^{\cong} \ar[d] &
M {\otimes}_A {\wh{A}} \ar[d] \\
{\stackrel{}{\underset {\gamma}{\varinjlim}} {\wh{M_{\gamma}}}} \ar[r] &
\wh{M}}
\]
The left vertical arrow is an isomorphism since each $M_{\gamma}$ is
finitely generated. Since each $M_{\gamma}$ is a submodule of $M$ and
since the completion functor is left exact,
we see that ${\wh{M_{\gamma}}} \inj \wh{M}$ for each $\gamma$ and hence
${\stackrel{}{\underset {\gamma}{\varinjlim}} {\wh{M_{\gamma}}}}
\inj \wh{M}$. A diagram chase now shows that the right vertical
arrow is injective.
\end{proof}
\section{Riemann-Roch for Higher Equivariant $K$-theory}
We begin this section with a brief recall of the equivariant higher
Chow groups defined in \cite{ED2}. Let $G$ a linear algebraic group over
$k$ acting on a quasi-projective variety $X$.
For any integer $j \ge 0$, let
$V$ be a representation of $G$ and let $U$ be a $G$-invariant open subset
of $V$ such that the codimension of the complement $V-U$ in $V$ is
larger than $j$, and $G$ acts freely on $U$. It is easy to see that
such representations of $G$ always exist. 
Let $X_G$ denote the quotient $X \stackrel{G} {\times} U$ of the product 
$X \times U$ by the diagonal action of $G$, which is free.
The equivariant higher Chow group $CH^j_G(X, i)$ is defined as the homology 
group $H_i\left(\sZ(X_G, \cdot)\right)$, where $\sZ(X_G, \cdot)$ is the 
Bloch's cycle complex of the variety $X_G$ \cite{Bloch}.
It is known ({\sl loc. cit.}) that this definition of $CH^j_G(X, i)$
is independent of the choice of the above representation. 
One should also observe that $CH^j_G(X, i)$ may be nonzero for infinitely 
many values of $j$, a crucial change from the non-equivariant higher Chow 
groups. We refer the reader to {\sl loc. cit.} and \cite{KV} for
further details.

Now we come back to the case of a connected reductive group $G$ over $k$. 
We fix a maximal torus $T = D_k(N)$ of $G$, where $N$ is the character group 
of $T$.
Let $\fm$ be a given maximal ideal of $R(G)$ and let $\wt{\fm}$ be a maximal
ideal of $R(T) = {\Q}[N]$ lying over $\fm$. Assume that the support 
$T_{\wt{\fm}}$ of $\wt{\fm}$ is connected. Until we begin the proof of 
Theorem~\ref{thm:TRR}, we also assume in this section that $S = 
{T_{\wt{\fm}}}_k$ is contained in the center of $G$.
Let $L = {R(T)}/{\wt{\fm}}$ be the residue field of $\wt{\fm}$.
Let $x$ and $y$ be the closed points of $X = {\rm Spec}\left(R(T)\right)$ 
and $Y = {\rm Spec}\left(R(G)\right)$ defined by $\wt{\fm}$ and $\fm$
respectively so that $f(x) = y$ under the finite map of finite type 
$\Q$-schemes $f : X \to Y$. Note that $X = D_{\Q}(N)$ is a split torus over 
$\Q$ such that $X_k \xrightarrow{\cong} D_k(N) = T$. 
Hence the natural map $X^*(X) \to X^*(T)$ is an isomorphism
of $\Q$-algebras. In the same way, we see that if $M$ is the character group
of $S$, then the natural map $X^*\left(T_{\wt{\fm}}\right) \to X^*(S)
={\Q}[M]$ is an isomorphism of $\Q$-algebras. 
We conclude that every character $\chi$ of $S$ is given by the morphism of
group schemes ${\chi} : T_{\wt{\fm}} \to {\G}_m$ over $\Q$. In particular, 
we get ${\chi}(x) \in L^*$.   

If $V$ is a representation of $G$, then $S$ acts on $V$ via characters
and this decomposes $V$ uniquely as
\begin{equation}\label{eqn:Dec}
V = \stackrel{} {\underset {\chi \in M}{\bigoplus}} V_{\chi},
\end{equation}
where $S$ acts on $V_{\chi}$ by $hv = {\chi}(h)v$. Since $S \subset Z(G)$,
we see that the above is a decomposition of $V$ as direct sum of
$G$-representations parameterized by $M$. 
This defines an $L$-algebra automorphism of $R_L(G)$ 
\begin{equation}\label{eqn:auto} 
t_{\wt{\fm}} = t_x : R_L(G) \to R_L(G)
\end{equation}
\[
t_{\wt{\fm}}([V]) = \stackrel{} {\underset {\chi \in M}{\Sigma}}
{\chi}(x)[V_{\chi}].
\]
The following lemma follows easily from above.
\begin{lem}\label{lem:twist}
The automorphism $t_{\wt{\fm}}$ takes the point $x$ of $Y_L$ to the closed 
point defined by the augmentation ideal $I_G$ of $R_L(G)$.
\end{lem}
\begin{proof}  We see from the above discussion that $t_{\wt{\fm}}$ acts
compatibly on ${T_{\wt{\fm}}}_L$ and $X_L$ such that the diagram
\[
\xymatrix{
{T_{\wt{\fm}}}_L \ar@{^{(}->}[r] \ar[d]_{t'_{\wt{\fm}}} &
X_L \ar[r]^{f} & Y_L \ar[d]^{t_{\wt{\fm}}} \\
{T_{\wt{\fm}}}_L \ar@{^{(}->}[r] &
X_L \ar[r]_{f} & Y_L}
\]
commutes. Since the augmentation ideal of $R(T_{\wt{\fm}})$
contracts to the augmentation ideal of $R(G)$, 
the lemma now follows immediately by noting that $x_L \in 
{T_{\wt{\fm}}}_L$ and by a direct checking in ~\ref{eqn:auto} that
$t'_{\wt{\fm}}(x_L) = (1, \cdots, 1) \in {T_{\wt{\fm}}}(L)$.
\end{proof}  

For a quasi-projective variety $X$ with a $G$-action, let $G^G(X)$
denote the $K$-theory spectrum of the abelian category of $G$-equivariant 
coherent sheaves on $X$. Assume that $S \subset G$ acts trivially on $X$.
We wish to decompose this spectrum in terms
of the characters of $S$. Let $\sC$ denote the abelian category of
$G$-equivariant coherent sheaves on $X$. Since $S$ acts trivially on $X$,
it acts on each coherent sheaf fiberwise. For any character $\chi$ of
$S$, let ${\sC}_{\chi}$ denote the subcategory of $\sC$ given by those
sheaves on which $S$ acts by the character $\chi$. That is, 
\[
{\rm ob} {\sC}_{\chi} = \left\{\sF \in \sC| sf = {\chi}(s)f \ \forall \ 
U \stackrel{} {\underset {\rm open}{\subset}} X, \forall f \in {\sF}(U)
\ {\rm and} \ \forall s \in S \right\}.
\]   
Note that since every open subset of $X$ is $S$-invariant, the above 
definition makes sense. 
\begin{prop}\label{prop:Dec1}
Each ${\sC}_{\chi}$ is an abelian full subcategory of $\sC$ and the inclusion
${\sC}_{\chi} \inj \sC$ induces a natural isomorphism of $K$-theory spaces
\begin{equation}\label{eqn:Dec2}
\stackrel{} {\underset {\chi \in M}{\coprod}} K\left({\sC}_{\chi}\right)
\xrightarrow{\cong} K^G(X).
\end{equation}
\end{prop}
\begin{proof} It is easy to see that ${\sC}_{\chi}$ is a full 
additive subcategory. To see that it has kernels and cokernels, one needs
to show that the kernel and cokernel of a map $\sF \to \sG$ in ${\sC}_{\chi}$
also lie in ${\sC}_{\chi}$, one can check it at each affine open subset of
$X$, where it can be checked directly. 
To prove ~\ref{eqn:Dec2}, note that $M = X^*(S)$ is a filtered category in a 
natural way by putting ${\rm Hom}(\chi, {\chi}')$ to be the empty set
unless $\chi = {\chi}'$ in which case, we take this to be the singleton
set consisting of the identity map. Then $M$ becomes a discrete filtered
category and there is a functor $F : M \to \underline{Cat}$ given by
$F(\chi) = {\sC}_{\chi}$. Now, since $S$ acts
trivially on $X$, every $\sF \in \sC$ has a unique decomposition
\[
\sF =  \stackrel{} {\underset {\chi \in M}{\coprod}} {\sF}_{\chi},
\ {\rm where}
\]
\[
{\sF}_{\chi}(U) = \left\{ f \in {\sF}(U) | sf = {\chi}(s)f \ \forall \
s \in S\right\}.  
\]
Furthermore, the condition $S \subset Z(G)$ shows that this is in fact a 
canonical decomposition of $\sF$ as a direct sum of $G$-equivariant coherent 
subsheaves parameterized by $M$.
This immediately shows that  
\[
\stackrel{} {\underset {\chi \in M}{\varinjlim}} F({\chi}) \cong
\stackrel{} {\underset {\chi \in M}{\coprod}} {\sC}_{\chi} 
\xrightarrow{\cong} \sC
\]
is an equivalence.
In particular, we get $\stackrel{} {\underset {\chi \in M}{\coprod}} 
BQ{\sC}_{\chi} \xrightarrow{\cong} BQ{\sC}$. The proposition now 
follows from \cite[Proposition~3]{Quillen}.
\end{proof}

We shall denote the space $K({\sC}_{\chi})$ also by $K^{\chi}(X)$.
We consider $L = {\Q}(x)$ as a discrete space and write $G^G(X)_L
= G^G(X) \Wedge L$ and $K^{\chi}(X)_L = K^{\chi}(X) \Wedge L$
as the smash products. Then we have ${\pi}_i\left(G^G(X)_L\right) =
G^G_i(X) {\otimes}_{\Z} L$ and ${\pi}_i\left(G^{\chi}(X)_L\right)
= K^{\chi}_i(X) {\otimes}_{\Z} L$.
This allows us to define an $L$-linear automorphism of 
$K^{\chi}_i(X)_L$ 
\begin{equation}\label{eqn:Dec2}
t^{\chi}_{\wt{\fm}} : K^{\chi}_i(X)_L \to K^{\chi}_i(X)_L
\end{equation}
\[
t^{\chi}_{\wt{\fm}}(\alpha) = {\chi}(x){\alpha}.
\]
This makes sense since ${\chi}(x) \in L^*$. By 
Proposition~\ref{prop:Dec1}, this induces  
an $L$-linear automorphism of $K^G_i(X)_L$ 
\begin{equation}\label{eqn:Dec3}
t_{\wt{\fm}} = \stackrel{} {\underset {\chi \in M}{\coprod}} 
t^{\chi}_{\wt{\fm}} : G^G_i(X)_L \to G^G_i(X)_L.
\end{equation}
We shall sometimes write $t_{\wt{\fm}}$ also as $t^G_{\wt{\fm}}$ if we need to 
indicate the underlying group.
It is easy to see from ~\ref{eqn:Dec2} and ~\ref{eqn:Dec3} that the
automorphism $t_{\wt{\fm}}$ coincides with the one defined
in ~\ref{eqn:auto} when $X = {\rm Spec}(k)$.
\begin{prop}\label{prop:match}
For every $\alpha \in R(G)_L$ and $\beta \in G^G_i(X)_L$, one has
\[
t_{\wt{\fm}}\left({\alpha}\cdot {\beta}\right)
= t_{\wt{\fm}}\left({\alpha}\right) \cdot  t_{\wt{\fm}}\left({\beta}\right)
\]
for the $R(G)_L$-module structure on $G^G_i(X)_L$ given by
$R(G)_L {\otimes}_L G^G_i(X)_L \to G^G_i(X)_L$.
\end{prop} 
\begin{proof} We first assume that $G = T$ is a torus. Then there is a
decomposition $T = S \times T'$ where $T' = T/S$. Since $S$ acts trivially
on $X$, we have the isomorphism $G^{T'}_i(X)_L {\otimes}_L R(S)_L 
\xrightarrow{\cong} G^T_i(X)_L$ by Lemma~\ref{lem:trivial*} below.
Moreover, since $x_k \in S$ and since
the characters of $S$ restrict to the trivial character on $T'$,
we see that $t^T_{\wt{\fm}}$ acts on $G^T_i(X)_L$ by
$t^T_{\wt{\fm}}\left({\beta}' \otimes {\alpha}' \right)
= {\beta}' \otimes t^S_{\wt{\fm}}({\alpha}')$.    
Writing $R(T) = R(T') \otimes R(S)$, ${\alpha} = {\alpha}_S \otimes
{\alpha}_{T'}$ and $\beta = {\beta}_S \otimes {\beta}_{T'}$, we get
\[
\begin{array}{lllll}
t^T_{\wt{\fm}}\left({\alpha} \cdot {\beta}\right) & = &
t^T_{\wt{\fm}}\left(\left({\alpha}_{T'} \cdot {\beta}_{T'}\right)
\otimes \left({\alpha}_{S} \cdot {\beta}_{S}\right)\right) & = &
{\alpha}_{T'} \cdot {\beta}_{T'} \otimes
t^S_{\wt{\fm}}\left({\alpha}_S \cdot {\beta}_S \right) \\
& = & {\alpha}_{T'} \cdot {\beta}_{T'} \otimes 
t^S_{\wt{\fm}}\left({\alpha}_S \right) \cdot 
t^S_{\wt{\fm}}\left({\beta}_S \right) & = &
\left({\alpha}_{T'} \otimes t^S_{\wt{\fm}}\left({\alpha}_S\right)\right)
\cdot \left({\beta}_{T'} \otimes t^S_{\wt{\fm}}\left({\beta}_S\right)\right) 
\\
& = & t^T_{\wt{\fm}}\left(\alpha\right) \cdot 
t^T_{\wt{\fm}}\left(\beta\right),
\end{array}
\]
where the first equality in the second row holds because $t_{\wt{\fm}}$
is a ring automorphism of $R(S)_L$ ({\sl cf.} Lemma~\ref{lem:twist}).
This proves the proposition in the case of a torus.

To prove the general case, let $p^! : G^G_i(X)_L \to G^T_i(X)_L$ be the
restriction map. We have seen before that $p^!$ is a split injective
$R(G)_L$-linear map ({\sl cf.} \cite[Section~1.6]{Thom2}). Moreover,  
as $x_k \in S$, it is clear from definition that $t_{\wt{\fm}}$ acts 
compatibly on $G^G(X)$ and $G^T(X)$. Now we have
\[
\begin{array}{lllll}
p^! \left(t^G_{\wt{\fm}}\left({\alpha}\cdot {\beta}\right)\right) & = &
t^T_{\wt{\fm}}\left(p^! \left({\alpha}\cdot {\beta}\right)\right) & = &
t^T_{\wt{\fm}}\left(p^! \left({\alpha}\right) \cdot
p^! \left({\beta}\right)\right) \\
& = & t^T_{\wt{\fm}}\left(p^! \left({\alpha}\right)\right) \cdot
t^T_{\wt{\fm}}\left(p^! \left({\beta}\right)\right) & = &
p^! \circ t^T_{\wt{\fm}}\left(\alpha\right) \cdot 
p^! \circ t^T_{\wt{\fm}}\left(\beta\right) \\
& = & p^! \left(t^T_{\wt{\fm}}\left(\alpha\right) \cdot 
t^T_{\wt{\fm}}\left(\beta\right)\right), &  & 
\end{array}
\]
where the first equality in the second row follows from the torus case
proved above. Since $p^!$ is injective, we get the desired result in the
general case.
\end{proof}
The following is a generalization of \cite[Lemma~5.6]{Thom3} where it
was proved when $X$ is affine.
\begin{lem}\label{lem:trivial*}
Let $T$ be a torus acting on a normal quasi-projective variety $X$ and let
$S \subset T$ be a subtorus acting trivially on $X$ so that $T$ acts on $X$
via $T' = T/S$. Then there is an isomorphism
\[
G^{T'}_i(X) {\otimes}_{\Q} R(S) \xrightarrow{\cong} G^T_i(X).
\]
\end{lem}
\begin{proof} When $X$ is affine, this was proved in \cite[Lemma~3.1]{Thom3}.
In general, as $X$ is normal, it can be covered by finitely many 
$T$-invariant affine open subsets ({\sl cf.} \cite[Corollary~3.11]{Sumihiro}).
Now the lemma is easily proved by an induction on the number of $T$-invariant
open subsets covering $X$ and the Mayer-Vietoris long exact sequence   
for the $K$-theory of $T$-equivariant coherent sheaves. As is well known,
this Mayer-Vietoris property is a direct consequence of the localization
and the excision properties, both of which hold for the $K$-theory of
equivariant coherent sheaves ({\sl cf.} 
\cite[Theorem~2.7, Corollary~2.4]{Thom4}).   
\end{proof}
\begin{cor}\label{cor:TRR**}
The automorphism $t_{\wt{\fm}}$ of $G^G_i(X)_L$ takes the 
$R(G)_L$-submodule ${{\fm}^j} G^G_i(X)_L$ to
${I_G^j} G^G_i(X)_L$ for all $j \ge 0$. In particular, it induces
an isomorphism
\[
\wh{t}_{\wt{\fm}} : {\wh{G^G_i(X)}}_{{\fm}, L} \xrightarrow{\cong}
{\wh{G^G_i(X)}}_{{I_G}, L}.
\]
\end{cor}
\begin{proof} This follows immediately from Proposition~\ref{prop:match}
since the action of $t_{\wt{\fm}}$ on $R(G)_L$ is by a ring automorphism
which takes the maximal ideal $\fm$ to the augmentation ideal $I_G$ by
Lemma~\ref{lem:twist}.
\end{proof} 

{\bf{Proof of Theorem~\ref{thm:TRR}:}}
Since $X$ is smooth, we can replace $G^G_i(X)$ with $K^G_i(X)$ and 
similarly for $X^{\wt{\fm}}$.
By Theorem~\ref{thm:NAL}, the natural map
\[
{\wh{K^G_i(X)}}_{\fm}  \xrightarrow{f^!_{\fm}}
{\wh{K^Z_i(X^{\wt{\fm}})}}_{\ov{\fm}}
\]
is an isomorphism of $\wh{R(G)}_{\fm}$-modules. Combining this with 
Corollary~\ref{cor:TRR**}, we have an isomorphism
\begin{equation}\label{eqn:TR*}
{\wh{K^G_i(X)}}_{{\fm}, L} \xrightarrow{\cong}
{\wh{K^Z_i(X^{\wt{\fm}})}}_{{I_Z}, L}.
\end{equation}
On the other hand, 
\cite[Theorem~1.2]{KV} implies that there is a Chern character map
\begin{equation}\label{eqn:TR*1}
{\wh{K^Z_i(X^{\wt{\fm}})}}_{{I_Z}, L} \xrightarrow{\wh{ch}}
\wh{CH^*_Z(X^{\wt{\fm}},i)}_L
\end{equation}
which is an isomorphism. The theorem now follows from ~\ref{eqn:TR*} and
~\ref{eqn:TR*1} by taking ${ch}^X_{\fm}$ to be 
${\wh{ch}} \circ {f^!_{\fm}}$.
$\hfill \square$
\\
\begin{ack}
This paper has been influenced and motivated in many ways by the work of 
Edidin and Graham \cite{ED3}. It came into being only after the author got 
aware of some of the questions Edidin and Graham had raised in their paper. 
We sincerely thank these authors for sharing these questions through their
paper.
\end{ack} 

School of Mathematics,
Tata Institute Of Fundamental Research, \\
Homi Bhabha Road,
Mumbai, 400005, India. \\ 
Email: amal@math.tifr.res.in  \\
\end{document}